\documentclass[oneside,american,english]{amsart}
\usepackage[T1]{fontenc}
\usepackage[latin9]{inputenc}
\usepackage{geometry}
\usepackage{mathrsfs}
\usepackage{amstext}
\usepackage{amsthm}
\usepackage{amssymb}

\makeatletter
\numberwithin{equation}{section}
\numberwithin{figure}{section}
\theoremstyle{plain}
\newtheorem{thm}{\protect\theoremname}
\theoremstyle{remark}
\newtheorem*{rem*}{\protect\remarkname}
\theoremstyle{plain}
\newtheorem{prop}[thm]{\protect\propositionname}
\theoremstyle{definition}
\newtheorem*{problem*}{\protect\problemname}
\theoremstyle{plain}
\newtheorem{lem}[thm]{\protect\lemmaname}
\theoremstyle{definition}
\newtheorem{defn}[thm]{\protect\definitionname}
\theoremstyle{remark}
\newtheorem{rem}[thm]{\protect\remarkname}
\theoremstyle{plain}
\newtheorem*{thm*}{\protect\theoremname}
\theoremstyle{plain}
\newtheorem{cor}[thm]{\protect\corollaryname}
\theoremstyle{remark}
\newtheorem{claim}[thm]{\protect\claimname}

\usepackage[all]{xy}

\makeatother

\usepackage{babel}
\addto\captionsamerican{\renewcommand{\claimname}{Claim}}
\addto\captionsamerican{\renewcommand{\corollaryname}{Corollary}}
\addto\captionsamerican{\renewcommand{\definitionname}{Definition}}
\addto\captionsamerican{\renewcommand{\lemmaname}{Lemma}}
\addto\captionsamerican{\renewcommand{\problemname}{Problem}}
\addto\captionsamerican{\renewcommand{\propositionname}{Proposition}}
\addto\captionsamerican{\renewcommand{\remarkname}{Remark}}
\addto\captionsamerican{\renewcommand{\theoremname}{Theorem}}
\addto\captionsenglish{\renewcommand{\claimname}{Claim}}
\addto\captionsenglish{\renewcommand{\corollaryname}{Corollary}}
\addto\captionsenglish{\renewcommand{\definitionname}{Definition}}
\addto\captionsenglish{\renewcommand{\lemmaname}{Lemma}}
\addto\captionsenglish{\renewcommand{\problemname}{Problem}}
\addto\captionsenglish{\renewcommand{\propositionname}{Proposition}}
\addto\captionsenglish{\renewcommand{\remarkname}{Remark}}
\addto\captionsenglish{\renewcommand{\theoremname}{Theorem}}
\providecommand{\claimname}{Claim}
\providecommand{\corollaryname}{Corollary}
\providecommand{\definitionname}{Definition}
\providecommand{\lemmaname}{Lemma}
\providecommand{\problemname}{Problem}
\providecommand{\propositionname}{Proposition}
\providecommand{\remarkname}{Remark}
\providecommand{\theoremname}{Theorem}

\begin{document}
\title{Algebraic independence and difference equations over elliptic function
fields}
\author{Ehud de Shalit}
\begin{abstract}
For a lattice $\Lambda$ in the complex plane, let $K_{\Lambda}$
be the field of $\Lambda$-elliptic functions. For two relatively
prime integers $p$ (respectively $q$) greater than 1, consider the
endomorphisms $\psi$ (resp. $\phi)$ of $K_{\Lambda}$ given by multiplication
by $p$ (resp. $q$) on the elliptic curve $\mathbb{C}/\Lambda$.
We prove that if $f$ (resp. $g$) are complex Laurent power series
that satisfy linear difference equations over $K_{\Lambda}$ with
respect to $\phi$ (resp. $\psi$) then there is a dichotomy. \emph{Either},
for some sublattice $\Lambda'$ of $\Lambda,$ one of $f$ or $g$
belongs to the ring $K_{\Lambda'}[z,z^{-1},\zeta(z,\Lambda')]$, where
$\zeta(z,\Lambda')$ is the Weierstrass zeta function, \emph{or} $f$
and $g$ are algebraically independent over $K_{\Lambda}.$ This is
an elliptic analogue of a recent theorem of Adamczewski, Dreyfus,
Hardouin and Wibmer (over the field of rational functions).
\end{abstract}

\date{July 27, 2022}
\address{Einstein Institute of Mathematics, The Hebrew Univeristy of Jerusalem}
\email{ehud.deshalit@mail.huji.ac.il}
\keywords{Difference equations, elliptic functions, algebraic independence}
\subjclass[2000]{12H10, 14H52, 39A10}
\maketitle

\section{Introduction}

\subsection{Background, over fields of rational functions}

A $\phi$-field is a field $K$ equipped with an endomorphism $\phi.$
The fixed field $C=K^{\phi}$ of $\phi$ is called \emph{the field
of constants }of $K.$ Throughout this paper we shall only consider
ground fields which are \emph{inversive: }$\phi$ is an automorphism
of $K,$ but for a general extension of $K$ we do not impose this
condition. Let $(K,\phi)\subset(F,\phi)$ be an extension of $\phi$-fields
(written from now on $K\subset F$), which is also inversive, and
with the same field of constants:
\[
C=K^{\phi}=F^{\phi}.
\]
Denote by $S_{\phi}(F/K)$ the collection of all $u\in F$ which satisfy
a linear homogenous $\phi$-difference equation
\begin{equation}
a_{0}\phi^{n}(u)+a_{1}\phi^{n-1}(u)+\cdots+a_{n}u=0,\label{eq:=00005Cphi-diff-eqn}
\end{equation}
with coefficients $a_{i}\in K.$ The set $S_{\phi}(F/K)$ is a $K$-subalgebra
of $F$.

Suppose now that $K$ and $F$ are endowed with a second automorphism
$\psi,$ commuting with $\phi,$ and that $\mathrm{tr.deg.}(K/C)\le1.$
Various results obtained in recent years support the philosophy that
\emph{if $\phi$ and $\psi$ are sufficiently independent, the $K$-algebras
$S_{\phi}(F/K)$ and $S_{\psi}(F/K)$ are also ``independent'' in
an appropriate sense.}

\bigskip{}

Here are some classical examples. Let $C$ be an algebraically closed
field of characteristic 0. We consider three classes of examples where
\begin{itemize}
\item (2S) $K=C(x),\,\,F=C((x^{-1})),\,\,\phi f(x)=f(x+h),\,\,\psi f(x)=f(x+k),$
($h,k\in C$),
\item (2Q) $K=C(x),\,\,F=C((x)),\,\,\phi f(x)=f(qx),\,\,\psi f(x)=f(px)$
($p,q\in C^{\times})$,
\item (2M) $K=\bigcup_{s=1}^{\infty}C(x^{1/s}),\,\,F=\bigcup_{s=1}^{\infty}C((x^{1/s}))$
(the field of Puiseux series), $\phi f(x)=f(x^{q}),\,\,\psi f(x)=f(x^{p})$
($p,q\in\mathbb{N}$).
\end{itemize}
Note that $\phi$ and $\psi$ are indeed \emph{auto}morphisms. In
all three examples, the assumption that $\phi$ and $\psi$ are ``sufficiently
independent'' is that the group 
\[
\Gamma=\left\langle \phi,\psi\right\rangle \subset\mathrm{Aut}(K)
\]
is free abelian of rank 2. The letters S,Q and M stand for ``shifts'',
``$q$-difference operators'' and ``Mahler operators'' respectively,
and the independence assumption gets translated into the additive
independence of $h$ and $k$ in the case (2S), and the multiplicative
independence of $p$ and $q$ in the cases (2Q) and (2M). Schäfke
and Singer proved in \cite{Sch-Si} the following theorem, confirming
the above philosophy.
\begin{thm}
\label{thm:Sch-Si theorem}Assume that $\Gamma$ is free abelian of
rank 2. Then in any of the three cases (2S), (2Q) or (2M)
\[
S_{\phi}(F/K)\cap S_{\psi}(F/K)=K.
\]
\end{thm}

Some instances of this theorem were known before. The case (2Q) dates
back to the work of Bézivin and Boutabaa \cite{Bez-Bou}, and the
case (2M), originally a conjecture of Loxton and van der Poorten \cite{vdPo},
was proved by Adamczewski and Bell in \cite{Ad-Be}. The earlier proofs,
however, used a variety of ad-hoc techniques, and only \cite{Sch-Si}
gave a unified treatment, revealing the common principles behind these
theorems. This new approach enabled the authors to prove a few more
theorems of the same nature, dealing with power series satisfying
simultaneously a $\phi$-difference equation and a linear ordinary
differential equation. See, also, the exposition in \cite{dS-G},
where we removed some unnecessary restrictions on the characteristic
of the field $C$ and on $|p|$ and $|q|$, in the case (2Q). In addition,
this last paper deals for the first time with a case, denoted there
(1M1Q), in which $\phi$ is a $q$-difference operator and $\psi$
a $p$-Mahler operator, and the resulting group $\Gamma$ is generalized
dihedral rather than abelian. The formulation of the analogous result
has to be cast now in the language of \emph{difference modules}, the
classical language of equations being inadequate when $\Gamma$ is
non-abelian. Modulo this remark, however, the main result \emph{and}
its proof are very similar, if not identical, to the above three cases.

\bigskip{}

The next breakthrough occured in a recent paper by Adamczewski, Hardouin,
Dreyfus and Wibmer \cite{A-D-H-W}. Building on an earlier work \cite{A-D-H}
dealing with difference-differential systems, these authors obtained
a far-reaching strengthening of the above theorem.
\begin{thm}
\label{thm:ADHW theorem}Consider any of the three cases (2S), (2Q)
or (2M). Let $f\in S_{\phi}(F/K)$ and $g\in S_{\psi}(F/K)$. If $f,g\notin K$
then $f$ and $g$ are algebraically independent over $K$.
\end{thm}

Letting $f=g$ one recovers Theorem \ref{thm:Sch-Si theorem}. The
key new tool which allows to upgrade Theorem \ref{thm:Sch-Si theorem}
to Theorem \ref{thm:ADHW theorem} is the ``parametrized'' Picard-Vessiot
theory, as developed in \cite{H-S,O-W}. We shall elaborate on this
theory and summarize its main ingredients in section \ref{sec:PPV}.

\subsection{\label{subsec:elliptic background}Background, over fields of elliptic
functions}

\subsubsection{The case (2Ell)}

In \cite{dS1,dS2} we initiated the study of the same theme over fields
of elliptic functions. For a lattice $\Lambda\subset\mathbb{C}$ let
$E_{\Lambda}$ stand for the elliptic curve whose associated Riemann
surface is $\mathbb{C}/\Lambda$ and
\[
K_{\Lambda}=\mathbb{C}(\wp(z,\Lambda),\wp'(z,\Lambda))
\]
its function field, the field of $\Lambda$-periodic meromorphic functions
on $\mathbb{C}.$ Fix $\Lambda_{0}$ and let
\[
K=\bigcup_{\Lambda\subset\Lambda_{0}}K_{\Lambda}.
\]
This is the function field of the universal cover of $E_{\Lambda_{0}},$
and should be compared to the field $K$ in the case (2M), which is
the function field of the universal cover of $\mathbb{G}_{m}$. Let
$p,q\in\mathbb{N}.$ Multiplication by $p$ or $q$ induces an endomorphism
of $E_{\Lambda}$ for each $\Lambda,$ and \emph{automorphisms} of
the field $K$ given by
\[
\phi f(z)=f(qz),\,\,\psi f(z)=f(pz).
\]
For $F$ we take the field of Laurent series $\mathbb{C}((z))$ with
the same $\phi$ and $\psi.$ Via the Taylor-Maclaurin expansion at
0, $K\subset F.$ We label this choice of $(K,F,\phi,\psi)$ by (2Ell).

\subsubsection{The ring $S$}

To formulate our results we need to introduce a ring slightly larger
than $K,$ namely the ring
\[
S=K[z,z^{-1},\zeta(z,\Lambda)]\subset F
\]
generated over $K$ by $z,$$z^{-1}$ and the Weierstrass zeta function
$\zeta(z,\Lambda).$ Recall that the latter is a primitive of $-\wp(z,\Lambda)$
and satisfies, for $\omega\in\Lambda$,
\[
\zeta(z+\omega,\Lambda)-\zeta(z,\Lambda)=\eta(\omega,\Lambda),
\]
where the additive homomorphism $\eta(\cdot,\Lambda):\Lambda\to\mathbb{C}$
is Legendre's eta function. It is easy to see that the ring $S$ does
not depend on which $\Lambda\subset\Lambda_{0}$ we use: once we adjoin
one $\zeta(z,\Lambda),$ they are all in $S$. It is also easy to
see that $\phi$ and $\psi$ induce automorphisms of $S$.

\subsubsection{Previous results in the case (2Ell)}

In \cite{dS2}, the following analogue of Theorem \ref{thm:Sch-Si theorem}
was proved:
\begin{thm}
\label{thm:dS2}Assume that $2\le p,q$ and $(p,q)=1$. Then in the
case (2Ell) we have
\[
S_{\phi}(F/K)\cap S_{\psi}(F/K)=S.
\]
\end{thm}

\begin{rem*}
(i) The reader should note the assumption on $p$ and $q$ being relatively
prime integers $\ge2.$ This is stronger than assuming $p$ and $q$
to be only multiplicatively independent. This stronger assumption
was needed in only one lemma of \cite{dS2}, but so far could not
be avoided.

(ii) The case (2Ell) brings up two completely new issues, absent from
the rational cases discussed so far. One is the issue of \emph{periodicity}.
The method of \cite{Sch-Si} starts with a formal analysis of the
solutions to our $\phi$- and $\psi$-difference equations at common
fixed points of $\phi$ and $\psi.$ Using estimates on coefficients
in Taylor expansions one shows that certain formal power series converge
in some open disks around these fixed points. Using the difference
equations one writes down a functional equation for these functions,
that allows to continue them meromorphically all the way up to a ``natural
boundary''. While each of the three cases (2S), (2Q) and (2M) has
its own peculiarities, and is technically different, the upshot in
all three cases is that a certain matrix with meromorphic entries
is proved to be \emph{globally} meromorphic on $\mathbb{P}^{1}(\mathbb{C}),$
hence a matrix with entries in $K=\mathbb{C}(x).$ This matrix is
used to descend a certain difference module attached to our system
of equations from $K$ to $\mathbb{C},$ and this leads to a proof
of Theorem \ref{thm:Sch-Si theorem}.

In the case (2Ell) the analysis of the situation starts along the
same lines. However, the matrix of globally meromorphic functions
on $\mathbb{C}$ thus produced bears, a priori, no relation to the
lattices $\Lambda.$ It starts its life as a matrix of formal power
series, convergent in some disk $|z|<\varepsilon,$ and is then continued
meromorphically using a functional equation with respect to $z\mapsto qz,$
losing the connection to the lattices. In fact, examples show that
this matrix \emph{need not be} a matrix of elliptic functions.

The Periodicity Theorem of \cite{dS1}, and its vast generalization
in \cite{dS2}, show that just enough of the periodicity can be salvaged
to push this approach to an end. A certain generalization of the ``baby
case'' of this theorem, considered in \cite{dS1}, will be instrumental
in the present work, when we deal with equations of the first order.

(iii) The second new issue in the case (2Ell) has to do with the emergence
of certain \emph{vector bundles} over the elliptic curve $E_{\Lambda}$,
that we associate to our system of difference equations. Luckily,
vector bundles over elliptic curves have been fully analyzed in Atiyah's
work \cite{At}. Their classification allows us to understand the
($\phi,\psi)$-difference modules associated to an $f\in S_{\phi}(F/K)\cap S_{\psi}(F/K)$.
The ensuing structure theorem for elliptic $(\phi,\psi)$-difference
modules is the main theorem of \cite{dS2}, and Theorem \ref{thm:dS2}
is a corollary of it. The need to include $\zeta(z,\Lambda)$ in $S$
reflects the non-triviality of these vector bundles. Over the field
$\mathbb{C}(z)$ none of this shows up, essentially because every
vector bundle over $\mathbb{G}_{a}$ or $\mathbb{G}_{m}$ is trivial.
\end{rem*}

\subsection{The main results}

Our main result is an elliptic analogue of Theorem \ref{thm:ADHW theorem}
(Theorem 1.3 of \cite{A-D-H-W}). In fact, both our result and Theorem
\ref{thm:ADHW theorem} admit a mild generalization. Let $AS_{\psi}(F/K)$
be the collection of all $u\in F$ for which there exists an $n\ge0$
such that
\[
\psi^{n}(u)\in K(u,\psi(u),\dots,\psi^{n-1}(u)).
\]
Clearly $S_{\psi}(F/K)\subset AS_{\psi}(F/K).$
\begin{thm}
\label{thm:Main Theorem A}Let $(K,F,\phi,\psi)$ be as in case (2Ell)
and assume that $2\le p,q$ and $(p,q)=1.$ Let $f\in S_{\phi}(F/K)$
and $g\in AS_{\psi}(F/K)$. If $f,g\notin S,$ then $f$ and $g$
are algebraically independent over $K$.
\end{thm}

The proof follows the strategy of \cite{A-D-H-W}. Theorem \ref{thm:Main Theorem A}
will be deduced from the following analogue of Theorem 4.1 there,
which concerns a \emph{single }power series\emph{ $f\in F.$}
\begin{thm}
\label{thm:Main Theorem B}Let $(K,F,\phi,\psi)$ be as in case (2Ell)
and assume that $2\le p,q$ and $(p,q)=1.$ Let $f\in S_{\phi}(F/K)$
and assume that $f\notin S.$ Then $\{f,\psi(f),\psi^{2}(f),\dots\}$
are algebraically independent over $K$.
\end{thm}

To explain the input from our earlier work, we have to formally introduce
the notion of a difference module, to which we already alluded several
times. A \emph{$\phi$-difference module }$(M,\Phi)$ over $K$ (called,
in short, a $\phi$-module) is a finite dimensional $K$-vector space
$M$ equipped with a $\phi$-linear bijective endomorphism $\Phi$.
Its \emph{rank }$\mathrm{rk}(M)$ is the dimension of $M$ as a $K$-vector
space. The set of $\Phi$-fixed points $M^{\Phi}$ is a $C$-subspace
of dimension $\le\mathrm{rk}(M).$

Since $\psi$ commutes with $\phi,$ the module
\[
M^{(\psi)}=(K\otimes_{\psi,K}M,1\otimes\Phi)
\]
is another $\phi$-module. Our $M$ is called $\psi$\emph{-isomonodromic
}(or $\psi$-integrable) if $M\simeq M^{(\psi)}$.

To any $\phi$-difference equation (\ref{eq:=00005Cphi-diff-eqn})
one can attach a $\phi$-module $M$ of rank $n$ whose fixed points
$M^{\Phi}$ correspond to the solutions of the equation in $K$. This
is classical, and explained in section \ref{subsec:eqns, systems, modules}
below. For this reason we shall refer to $M^{\Phi}$ also as the space
of ``solutions'' of $M.$

To any $\phi$-module $M$ of rank $n$ over $K$ one can associate
a \emph{difference Galois group} $\mathcal{G}$, which is a Zariski
closed subgroup of $GL_{n,\mathbb{C}}$, uniquely determined up to
conjugation (and reviewed in section \ref{subsec:The-difference-Galois}
below). This linear algebraic group measures the algebraic relations
that exist between the solutions of $M$, not over $K$ itself (where
there might be none, or too few solutions), but after we have base-changed
to a suitable universal extension - the Picard-Vessiot extension -
in which a full set of solutions can be found. The larger $\mathcal{G}$
is, the fewer such relations exist. The analogy with classical Galois
theory, in which the Galois group measures the algebraic relations
between the roots of a polynomial in a splitting field, is obvious.

The input, deduced from the main theorem of \cite{dS2}, needed in
the proof of Theorem \ref{thm:Main Theorem B}, is the following.
We continue to assume that $2\le p,q$ and $(p,q)=1.$
\begin{thm}
\label{thm:isomonodromy and solvability}Assume that $M$ is $\psi$-isomonodromic.
Then its difference Galois group $\mathcal{G}$ is solvable.
\end{thm}

In addition, we shall need a generalization of the ``baby'' Periodicity
Theorem of \cite{dS1}, explained in section \ref{subsec:A-periodicity-theorem}.

\bigskip{}

Except for these two results, the rest of the proof of Theorems \ref{thm:Main Theorem A}
and \ref{thm:Main Theorem B} imitates \cite{A-D-H-W}. As this depends
on results scattered through many references \cite{A-D-H,A-D-H-W,D-H-R,DV-H-W1,DV-H-W2,O-W},
we shall make an effort to collect all the prerequisites in a way
that facilitates the reading.

\subsection{Outline of the paper}

Section 2 will be devoted to generalities on difference equations,
difference modules, Picard-Vessiot extensions and the difference Galois
group. The standard reference here is \cite{S-vdP}, although our
language will sometimes be different.

Section 3 will be devoted to the more recent theory of \emph{parametrized}
Picard-Vessiot theory and the \emph{parametrized }difference Galois
group, to be found in the references cited above.

In Section 4 we shall prove the two results that we need as input
in the proof of Theorem \ref{thm:Main Theorem B}, relying on \cite{dS1,dS2}.

Section 5 will start by explaining how to deduce Theorem \ref{thm:Main Theorem A}
from Theorem \ref{thm:Main Theorem B}. We shall then carry out the
proof of Theorem \ref{thm:Main Theorem B}, following the program
of \cite{A-D-H-W}.

\section{Review of classical Picard-Vessiot theory}

\subsection{The ground field}

Let $K$ be defined as in section \ref{subsec:elliptic background},
in the case (2Ell). We shall need the following facts about it.
\begin{prop}
\label{prop:Properties of K}(i) $K$ is a $C^{1}$ field (any homogenous
polynomial of degree $d$ in $n>d$ variables has a nontrivial zero
in $K$).

(ii) If $G$ is a connected linear algebraic group over $K$ then
any $G$-torsor over $K$ is trivial.

(iii) $K$ does not have any non-trivial finite extension $L/K$ to
which $\phi$ (or $\psi$) extends as an automorphism.
\end{prop}

\begin{proof}
(i) It is enough to prove the claim for every $K_{\Lambda}$, where
this is Tsen's theorem: the function field of any curve over an algebraically
closed field of characteristic 0 is a $C^{1}$-field.

(ii) This is Springer's theorem: a $C^{1}$-field of characteristic
0 is of cohomological dimension $\le1$. By Steinberg's theorem this
implies that every torsor of a connected linear algebraic group over
$K$ is trivial. See \cite{Se} ch. III.2.

(iii) (Compare \cite{D-R}, Proposition 6). Suppose $L$ is a finite
extension of $K$ to which $\phi$ extends as an automorphism. Then,
for $\Lambda$ small enough, $L=L_{\Lambda}K$ where $L_{\Lambda}$
is an extension of $K_{\Lambda}$, $[L:K]=[L_{\Lambda}:K_{\Lambda}]$.
Let $\Lambda'\subset\Lambda$ and $L_{\Lambda'}=L_{\Lambda}K_{\Lambda'}.$
Then for $\Lambda'$ sufficiently small $\psi(L_{\Lambda})\subset L_{\Lambda'}$.
Replacing $\Lambda$ by $\Lambda'$ we may therefore assume that $\psi(L_{\Lambda})\subset L_{\Lambda}.$
Thus $\psi$ extends to an endomorphism of $L_{\Lambda}.$ Let $\pi:Y\to E_{\Lambda}$
be the covering of complete nonsingular curves corresponding to $L_{\Lambda}\supset K_{\Lambda}$
and $\alpha:Y\to Y$ the morphism inducing $\psi$ on $L_{\Lambda}.$
Since $\pi\circ\alpha=[q]\circ\pi$ we get that $\deg(\alpha)=q^{2}.$
By the Riemann-Hurwitz formula
\[
2g_{Y}-2=(2g_{Y}-2)q^{2}+\sum_{x\in Ram(\alpha)}(e_{x}-1)
\]
where $g_{Y}\ge1$ is the genus of $Y$ and $Ram(\alpha)$ the ramification
locus of $\alpha$, $e_{x}$ being the ramification index. This equation
can only hold if $g_{Y}=1$ (and $\alpha$ is everywhere unramified).
In particular, $\pi$ is an isogeny of elliptic curves, hence $L_{\Lambda}\subset K$
and $L=K.$
\end{proof}
Define
\[
S=K[z,z^{-1},\zeta(z,\Lambda)]\subset F.
\]
If $\Lambda'\subset\Lambda$ is another lattice then
\[
\wp(z,\Lambda)-\sum_{\omega\in\Lambda/\Lambda'}\wp(z+\omega,\Lambda')
\]
is a meromorphic $\Lambda$-periodic function. Its poles are contained
in $\Lambda$, but at 0 the poles of $\wp(z,\Lambda)$ and of $\wp(z,\Lambda')$
cancel each other, while the other terms have no pole. It follows
that this $\Lambda$-periodic function has no poles, hence is a constant.
Integrating, we find that
\[
\zeta(z,\Lambda)-\sum_{\omega\in\Lambda/\Lambda'}\zeta(z+\omega,\Lambda')=az+b
\]
for some $a,b\in\mathbb{C}.$ On the other hand $\zeta(z+\omega,\Lambda')-\zeta(z,\Lambda')\in K_{\Lambda'}\subset K.$
It follows that
\[
\zeta(z,\Lambda)-[\Lambda:\Lambda']\zeta(z,\Lambda')\in K[z,z^{-1}].
\]
This shows that the definition of $S$ does not depend on which $\Lambda\subset\Lambda_{0}$
we use. Since for any rational number $r=m/n$ $\zeta(rz,\Lambda)-r\zeta(z,\Lambda)\in K_{n\Lambda}\subset K,$
$\phi$ and $\psi$ induce automorphisms of $S$.
\begin{problem*}
Does the field of fractions of $S$ satisfy Proposition \ref{prop:Properties of K}?
\end{problem*}

\subsection{\label{subsec:eqns, systems, modules}Difference equations, difference
systems and difference modules}

In this subsection and the next ones, the $\phi$-field $(K,\phi)$
can be arbitrary. The standard reference is \cite{S-vdP}. As usual,
to the difference equation (\ref{eq:=00005Cphi-diff-eqn}) we associate
the companion matrix
\[
A=\left(\begin{array}{ccccc}
0 & 1\\
 & 0 & 1\\
 &  & \ddots\\
 &  &  & 0 & 1\\
-a_{n}/a_{0} &  & \cdots &  & -a_{1}/a_{0}
\end{array}\right),
\]
and the first order linear system of equations
\begin{equation}
\phi(Y)=AY\label{eq:system}
\end{equation}
for which we seek solutions $Y=\,^{t}(u_{1},\dots,u_{n})$ in $\phi$-ring
extensions $L$ of $K$. Notice that if $u$ is a solution of (\ref{eq:=00005Cphi-diff-eqn})
then $^{t}(u,\phi(u),\dots,\phi^{n-1}(u))$ is a solution of (\ref{eq:system}). 

From now on we concentrate on first order systems of equations of
the form (\ref{eq:system}) with $A\in GL_{n}(K)$ \emph{arbitrarily
given}.

With the system (\ref{eq:system}) we associate the $\phi$-difference
module $M=(K^{n},\Phi)$ where
\[
\Phi(v)=A^{-1}\phi(v).
\]
Notice that a solution $v\in K^{n}$ to (\ref{eq:system}) is nothing
but an element of $M^{\Phi},$ the fixed points of $\Phi$ in $M$.
This is a $C=K^{\phi}$-subspace, and the well-known Wronskian Lemma
shows that
\[
\dim_{C}M^{\Phi}\le\mathrm{rk}M.
\]
By abuse of language, we shall refer to $M^{\Phi}$ also as the space
of ``solutions'' of $M$.

An equality $\dim_{C}M^{\Phi}=\mathrm{rk}M$ holds if and only if
a full set of solutions of (\ref{eq:system}) exists in $K$, if and
only if $M$ is isomorphic to the trivial module $(K^{n},\phi).$
In such a case a matrix $U\in GL_{n}(K)$ satisfying
\[
\phi(U)=AU
\]
is called a \emph{fundamental matrix }for (\ref{eq:system}). Its
columns form a basis of $M^{\Phi}$ over $C$. 

Given any $\phi$-module, a choice of a basis of $M$ over $K$ shows
that it is of the above form. A choice of another basis results in
a \emph{gauge transformation}, replacing $A$ with
\[
A'=\phi(P)^{-1}AP
\]
for some $P\in GL_{n}(K).$ Conversely, the systems of equations defined
by $A$ and by $A'$ are equivalent if and only if $A$ and $A'$
are gauge equivalent. The transition from a system of equations to
a $\phi$-module can therefore be reversed. Thanks to Birkhoff's cyclicity
lemma, the transition from a single linear equation of order $n$
to a system of equations of order one can also be reversed. The three
notions are therefore equivalent, and which language one chooses to
work with is very much a matter of taste.

\subsection{Isomonodromy and $(\Phi,\Psi)$-difference modules}

Assume now that the $\phi$-field $K$ is endowed with a second automorphism
$\psi$, \emph{commuting} with $\phi$. If $(M,\Phi)$ is a $\phi$-module
over $K$, then
\[
M^{(\psi)}=(K\otimes_{\psi,K}M,1\otimes\Phi)
\]
is another $\phi$-module, called the $\psi$\emph{-transform} of
$M$. If $M=(K^{n},\Phi)$ with $\Phi(v)=A^{-1}\phi(v)$ then $M^{(\psi)}$
is likewise given by the matrix $\psi(A).$

The notion of a $(\phi,\psi)$-difference module is naturally defined.
It is a finite dimensional vector space $M$ over $K$ equipped with
bijective $\phi$-linear (resp. $\psi$-linear) endomorphisms $\Phi$
(resp. $\Psi$) commuting with each other: $\Phi\circ\Psi=\Psi\circ\Phi$.
\begin{lem}
For a $\phi$-module $M$ over $K$ of rank $n$, the following are
equivalent:

(i) $M^{(\psi)}\simeq M$ as a $\phi$-module.

(ii) $M$ admits a structure of a $(\phi,\psi)$-module extending
the given $\phi$-module structure.

(iii) If $A=A_{\phi}$ is the matrix associated to $M$ in some basis,
there exists a matrix $A_{\psi}\in GL_{n}(K)$ satisfying the compatibility
condition
\[
\phi(A_{\psi})A_{\phi}=\psi(A_{\phi})A_{\psi}.
\]
\end{lem}

The proof is left as an easy exercise. A $\phi$-module satisfying
the above conditions is called $\psi$-\emph{isomonodromic} (or $\psi$-integrable).
Property (iii) shows that the definition is symmetric: If $\left(M,\Phi\right)$
is $\psi$-isomonodromic and $\Psi$ is the $\psi$-linear operator
as in (ii), then $(M,\Psi)$ is $\phi$-isomonodromic as a $\psi$-module.
The terminology is derived from the differential set-up, of which
the theory of difference equations is a discrete analogue.

\subsection{Picard-Vessiot theory}

\subsubsection{Picard-Vessiot rings and extensions}

It is natural to look for an extension of $K$ in which (\ref{eq:system})
attains a full set of solutions, or, equivalently, over which the
associated module $M$ is trivialized, after base change. Easy examples
show that such an extension might have to have zero divisors. The
best we can do is encapsulated in the following definition.
\begin{defn}
(i) A $\phi$\emph{-ring} is a commutative unital ring $R$ equipped
with an \emph{endomorphism} $\phi$. It is called $\phi$-\emph{simple}
if it does not have any non-zero ideals $I$ invariant under $\phi$,
i.e. satisfying $\phi(I)\subset I.$

(ii) A \emph{Picard-Vessiot (PV) ring} for the $\phi$-module $M$
(associated to $A\in GL_{n}(K)$ as above) is a simple $\phi$-ring
extension $(R,\phi)$ of $(K,\phi)$ over which $M_{R}=(R\otimes_{K}M,\phi\otimes\Phi)$
is trivialized (i.e. becomes isomorphic to $(R^{n},\phi)$), and such
that $R=K[u_{ij},\det(U)^{-1}]$ if $U=(u_{ij})\in GL_{n}(R)$ is
a fundamental matrix of (\ref{eq:system}).
\end{defn}

Here are the main properties of PV rings.
\begin{itemize}
\item PV rings exist, are noetherian and (like any $\phi$-simple ring)
reduced. Furthermore, a PV ring $R$ is a finite product $R_{1}\times\cdots\times R_{t}$
of integral domains, permuted cyclically by $\phi.$
\item Since $\phi$ was assumed to be an automorphism of $K$ and $A$ is
invertible, a PV ring $R$ happens to be \emph{inversive: $\phi$
}is an automorphism of $R$.
\item The field of constants $C_{R}=R^{\phi}$ is an algebraic extension
of $C=K^{\phi}.$ If $C$ is algebraically closed, $C=C_{R}.$
\item The fundamental matrix $U\in GL_{n}(R)$ is unique up to $U\mapsto UV$
with $V\in GL_{n}(C_{R}).$
\item If $C$ is algebraically closed, any two PV rings for $M$ are (noncanonically)
isomorphic.
\item Let $L=Quot(R)$ be the total ring of fractions of $R$ (the localization
of $R$ in the $\phi$-invariant multiplicative set of non-zero divisors
of $R$). Thus $L=L_{1}\times\cdots\times L_{t}$ is a finite product
of fields, which are permuted cyclically by $\phi$. We have $L^{\phi}=C_{R}.$
A $\phi$-ring $L$ of this type is a called a $\phi$\emph{-pseudofield}.
\end{itemize}
Assume from now on that $C$ is algebraically closed.
\begin{lem}
Let $L$ be a $\phi$-pseudofield extension of $K$ which trivializes
$M$ and is generated over $K$ (as a pseudofield) by the entries
$u_{ij}$ of a fundamental matrix $U.$ Suppose that $L^{\phi}=C.$
Then $R=K[u_{ij},\det(U)^{-1}]\subset L$ is $\phi$-simple, hence
it is a PV ring for $M$, and $L$ is its total ring of fractions.
\end{lem}

The last lemma is of great practical value, because it is often much
easier to check that $L^{\phi}=C$ then to verify directly that $R$
is $\phi$-simple. The $\phi$-pseudofield $L$ is called the PV\emph{
extension} associated with $M$.
\begin{itemize}
\item Notation as above, $L_{1}$ is a $\phi^{t}$-PV extension for $(M,\Phi^{t})$
over $(K,\phi^{t}).$ Note that the matrix associated to $(M,\Phi^{t})$
is 
\[
A_{[t]}=\phi^{t-1}(A)\cdots\phi(A)A.
\]
Thus, at the expense of replacing $\phi$ by a suitable power, we
may assume that $L$ is a field and $R$ a domain. In the current
paper, this will turn out to be always possible.
\end{itemize}
A PV ring $R$ for (\ref{eq:system}) is constructed as follows. Let
$X=(X_{ij})$ be an $n\times n$ matrix of indeterminates. Let $\phi$
act on the ring $\widetilde{R}=K[X_{ij},\det(X)^{-1}]$ via its given
action on $K$ and the formula
\[
\phi(X)=AX,
\]
i.e. $\phi(X_{ij})=\sum_{\nu=1}^{n}a_{i\nu}X_{\nu j}$. Let $I$ be
a maximal $\phi$-invariant ideal in $\widetilde{R}$. Then
\[
R=\widetilde{R}/I
\]
is a PV ring for (\ref{eq:system}), and $U=X\mod I$ is a fundamental
matrix in $GL_{n}(R).$ We remark that since $\widetilde{R}$ is noetherian,
any $\phi$-invariant ideal $I$ satisfies $\phi(I)=I.$

The reduced $K$-scheme $W=Spec(R)$ is called the PV\emph{ scheme}
associated with $R.$ Since the choice of a fundamental matrix $U$
amounts to a presentation $R=\widetilde{R}/I$ as above, the choice
of $U$ determines a closed embedding
\[
W\hookrightarrow GL_{n,K}.
\]

In general, the $K$-scheme $W$ might not have any $K$-points. We
shall see soon (Proposition \ref{prop:torsor is trivial}) that if
$K$ satisfies the conclusions of Proposition \ref{prop:Properties of K},
$W(K)\ne\emptyset$. This will be an important observation in our
context. 

\subsubsection{\label{subsec:The-map =00005Ctau}The map $\tau$}

If $h\in GL_{n}(K)=\mathrm{Hom}_{K}(K[X_{ij},\det(X)^{-1}],K)$ we
let
\[
\tau(h)=\phi\circ h\circ\phi^{-1}\in GL_{n}(K).
\]
If $X_{h}=h(X)$ is the matrix in $GL_{n}(K)$ representing the $K$-point
$h$, then since $\phi^{-1}(X)=\phi^{-1}(A)^{-1}X$ we have
\[
X_{\tau(h)}=\tau(h)(X)=A^{-1}\phi(X_{h}).
\]
If $h\in W(K),$ i.e. $h$ factors through $I$, then since $\phi^{-1}(I)=I,$
so does $\tau(h).$ Regarded as a subset of $GL_{n}(K),$ if $P\in W(K)$
then
\[
\tau(P)=A^{-1}\phi(P)\in W(K).
\]
The set of $K$-points of the Picard-Vessiot scheme is therefore,
if not empty, invariant under $\tau.$

\subsection{\label{subsec:The-difference-Galois}The difference Galois group
of $(M,\Phi)$}

We continue to assume that $C=K^{\phi}$ is algebraically closed.
Let $(M,\Phi)$ be a $\phi$-module, $A$ the matrix associated to
it in some basis, $R$ a PV ring and $L=Quot(R)$ the associated PV
extension.

Let $B$ be a $C$-algebra, with a trivial $\phi$-action. Writing
$-_{B}=B\otimes_{C}-$ we let
\[
\mathcal{G}(B)=\mathrm{Aut}_{\phi}(R_{B}/K_{B})=\mathrm{Aut}_{\phi}(L_{B}/K_{B})
\]
be the group of automorphisms of $R_{B}$ that fix $K_{B}$ pointwise
and commute with $\phi$. This yields a functor
\[
\mathcal{G}:\mathrm{Alg}_{C}\rightsquigarrow\mathrm{Groups}.
\]
Then:
\begin{itemize}
\item $\mathcal{G}$ is representable by a closed subgroup scheme of $GL_{n,C}.$
If $\sigma\in\mathcal{G}(B)$
\[
\sigma(U)=U\cdot V(\sigma)
\]
with $V(\sigma)\in GL_{n}(B)$ and $\sigma\mapsto V(\sigma)$ embeds
$\mathcal{G}$ in $GL_{n,C}.$ If $char(C)=0$ then $\mathcal{G}$
is reduced, but in positive characteristic we must include the possibility
of non-reduced $\mathcal{G}.$
\item If $A$ is replaced by $\phi(P)^{-1}AP$ (change of basis of $M$,
$P\in GL_{n}(K)$) and $U$ is replaced by $P^{-1}U$ then, since
$\sigma(P)=P$, we get \emph{the same embedding} $\mathcal{G}\hookrightarrow GL_{n,C}.$
If $U$ is replaced by another fundamental matrix for (\ref{eq:system}),
necessarily of the form $UT$ with $T\in GL_{n}(C),$ then $V(\sigma)$
is replaced by $T^{-1}V(\sigma)T.$ Thus $\mathcal{G}$ \emph{is uniquely
determined up to conjugation in} $GL_{n}(C).$
\item The coordinate ring of $\mathcal{G}$ is given by $C[\mathcal{G}]=(R\otimes_{K}R)^{\phi}$.
Let
\[
Z=(U^{-1}\otimes1)\cdot(1\otimes U)\in GL_{n}(R\otimes_{K}R),
\]
i.e. $Z_{ij}=\sum_{\nu=1}^{n}(U^{-1})_{i\nu}\otimes U{}_{\nu j}$.
Then
\[
\phi(Z)=(U^{-1}\otimes1)\cdot(A^{-1}\otimes1)\cdot(1\otimes A)\cdot(1\otimes U)=Z
\]
so $Z_{ij}\in C[\mathcal{G}]$ and $C[\mathcal{G}]=C[Z_{ij},\det Z^{-1}].$
We have
\[
\sigma\in\mathcal{G}(B)=\mathrm{\mathrm{Hom}}(C[\mathcal{G}],B)\leftrightsquigarrow(Z\mapsto V(\sigma))
\]
and $\mathcal{G}\hookrightarrow GL_{n,C}$ implies the comultiplication
\[
m^{*}(Z)=Z\otimes Z,
\]
i.e. $m^{*}(Z_{ij})=\sum_{\nu=1}^{n}Z_{i\nu}\otimes Z_{\nu j}$.
\item Inside $R\otimes_{K}R$ we have the canonical isomorphism
\[
R\otimes_{K}K[\mathcal{G}]=R\otimes_{C}C[\mathcal{G}]\simeq R\otimes_{K}R
\]
(since $(U\otimes1)\cdot Z=1\otimes U$), which means that $W=Spec(R)$
is a \emph{torsor} of $\mathcal{G}_{K}.$ We conclude that $W(K)\ne\emptyset$
is a necessary and sufficient condition for $W$ to be the trivial
torsor, i.e. to be (noncanonically) isomorphic to $\mathcal{G}_{K}.$
\item If $L$ is a field, $\mathrm{tr.deg.}L/K=\dim\mathcal{G}.$
\end{itemize}
\begin{prop}
\label{prop:torsor is trivial}Assume that $char(C)=0$ and that $K$
satisfies the conclusions of Proposition \ref{prop:Properties of K}
for every power of $\phi$. Then $W(K)\ne\emptyset$.
\end{prop}

\begin{proof}
If $\mathcal{G}$ is connected, this follows from part (ii) of Proposition
\ref{prop:Properties of K}. Following Proposition 1.20 in \cite{S-vdP}
we explain how part (iii) of the same Proposition allows us to get
rid of the assumption that $\mathcal{G}$ is connected. Let
\[
R=R_{1}\times\cdots\times R_{t}
\]
be the decomposition of $R$ into a product of integral domains, permuted
cyclically by $\phi$. Since $K$ does not have any finite extension
to which $\phi^{t}$ extends, it is algebraically closed in the field
$L_{i}=Quot(R_{i})$. This means that $W_{i}=Spec(R_{i})$ remains
irreducible over the algebraic closure $\overline{K}$ of $K$. It
follows that one of the $W_{i}$, say $W_{1},$ is a torsor of $\mathcal{G}_{K}^{0}.$
Since $\mathcal{G}_{K}^{0}$ is connected, $W_{1}(K)\ne\emptyset$,
hence $W(K)\ne\emptyset.$
\end{proof}
We continue to assume that $C$ is algebraically closed of characteristic
0.
\begin{thm}
\label{thm:structure group and the Galois group}(\cite{S-vdP} Theorem
I.1.21) (i) Let $H\subset GL_{n,C}$ be a closed subgroup. If in some
basis $A\in H(K),$ then we can choose $U\in H(R)$ and $\mathcal{G}\subset H$.
For a general fundamental matrix $U\in GL_{n}(R),$ some conjugate
of $\mathcal{G}$ by an element of $GL_{n}(C)$ will be contained
in $H$.

(ii) Assume that the conclusions of Proposition \ref{prop:Properties of K}
hold. Then conversely, there exists a basis of $M$ with respect to
which $A\in\mathcal{G}(K).$ Equivalently, for the original $A$ there
exists a $P\in GL_{n}(K)$ such that $\phi(P)^{-1}AP\in\mathcal{G}(K).$

(iii) Under the assumptions of (ii) $\mathcal{G}$ is characterized
(up to conjugation by $GL_{n}(C)$) as a \emph{minimal} element of
the set
\[
\mathscr{H}=\{H\subset GL_{n,C}|\,H\,\mathrm{closed,}\,\exists P\in GL_{n}(K)\,\mathrm{s.t.}\,\,\,\phi(P)^{-1}AP\in H(K)\}.
\]
Every other element of $\mathscr{H}$ therefore contains a conjugate
of $\mathcal{G}$.

(iv) $\mathcal{G}/\mathcal{G}^{0}$ is cyclic.
\end{thm}

\begin{proof}
(i) Assume $A\in H(K)$, and $H$ is given explicitly as $Spec(C[X_{ij},\det(X)^{-1}]/N)$
where $N$ is a Hopf ideal. Let $\phi$ act on $K[X_{ij},\det(X)^{-1}]$
via the given action on $K$ and via $\phi(X)=AX$. Then $N_{K}$
is a $\phi$-ideal because if $f\in N$ then
\[
\phi(f)=(\alpha\times1)\circ m^{*}(f)
\]
where $m^{*}$ is the comultiplication and $\alpha$ the homomorphism
$C[X_{ij},\det(X)^{-1}]\to K$ substituting $A$ for $X$. But 
\[
m^{*}(f)\in N\otimes_{C}C[X_{ij},\det(X)^{-1}]+C[X_{ij},\det(X)^{-1}]\otimes_{C}N
\]
and $\alpha(N)=0$ since $A\in H(K).$ Thus $\phi(f)\in K\otimes_{C}N=N_{K}$.

Let $I$ be a maximal $\phi$-ideal in $K[X_{ij},\det(X)^{-1}]$ \emph{containing}
$N_{K}$ and $U=X\mod I$. Then
\[
W=Spec(R)=Spec(K[X_{ij},\det(X)^{-1}]/I)\subset H_{K}
\]
is the PV scheme, and $U\in H(R)$ (corresponding to the canonical
homomorphism 
\[
C[X_{ij},\det(X)^{-1}]/N\to K[X_{ij},\det(X)^{-1}]/I.)
\]
It follows that for $\sigma\in\mathrm{Aut_{\phi}}(R_{B}/K_{B})$ we
have $\sigma(U)\in H(R_{B})$, hence
\[
V(\sigma)=U^{-1}\sigma(U)\in H(R_{B})^{\phi}=H(B).
\]
Any other fundamental matrix is of the form $UT$ with $T\in GL_{n}(C)$,
so $T\mathcal{G}T^{-1}\subset H.$

(ii) Under our assumptions, $W(K)$ is non-empty. Any $P\in W(K)\subset GL_{n}(K)$
satisfies
\[
W=P\mathcal{G}_{K}.
\]
Since $\tau(P)\in W(K)$ as well (see \ref{subsec:The-map =00005Ctau}),
$\tau(P)^{-1}P=\phi(P)^{-1}AP\in\mathcal{G}(K)$ and there exists
a basis of $M$ for which $A\in\mathcal{G}(K).$

(iii) By (i) every member of $\mathscr{H}$ contains $\mathcal{G}$
up to conjugacy. By (ii) every $\mathcal{G}$ (uniquely determined
up to conjugacy) belongs to $\mathscr{H}.$ Thus $\mathcal{G}$ is
the unique minimal member of $\mathscr{H},$ up to conjugacy. Note
that it is not a-priori clear that all the minimal members of $\mathscr{H}$
are conjugate, but this follows from the proof.

(iv) See \cite{S-vdP}.
\end{proof}

\subsection{The Galois correspondence}

Let $R$ be a PV ring for $M$ and $L$ its total ring of fractions.
Let $\mathcal{G}$ be the difference Galois group of $M.$

We quote the following basic theorem. We say that $a\in L$ is fixed
by $\mathcal{G}$ and write $a\in L^{\mathcal{G}}$ if for every $B\in\mathrm{Alg}_{C}$
the element $1\otimes a\in L_{B}$ is fixed by $\mathcal{G}(B).$
The set $L^{\mathcal{G}}$ is a $\phi$-sub-pseudofield of $L$.
\begin{thm}
(i) For any closed subgroup $\mathcal{G}'\subset\mathcal{G}$ let
\[
\text{\ensuremath{\mathscr{F}(\mathcal{G}')=K'=}}L^{\mathcal{G}'},
\]
a $\phi$-sub-pseudofield of $L$ containing $K$. For any $\phi$-pseudofield
$K\subset K'\subset L$ let $\mathscr{G}(K')=\mathcal{G}'$ be the
closed subgroup of $\mathcal{G}$ whose $B$-points, for $B\in\mathrm{Alg}_{C}$
are 
\[
\mathscr{G}(K')(B)=\mathcal{G}'(B)=\mathrm{Aut_{\phi}}(L_{B}/K'_{B}).
\]
Then $\mathcal{G}'\mapsto K'=\mathscr{F}(\mathcal{G}')$ and $K'\mapsto\mathcal{G}'=\mathscr{G}(K')$
is a 1-1 correspondence between closed subgroups of $\mathcal{G}$
and $\phi$-sub-pseudofields of $L$ containing $K$. In particular
$L^{\mathcal{G}}=K.$

(ii) $\mathcal{G}'$ is normal in $\mathcal{G}$ if and only if $K'$
is a PV extension of some difference $\phi$-module $M'.$ In this
case the difference Galois group of $K'/K$ is $\mathcal{G}/\mathcal{G}'$.

(iii) $\mathcal{G}^{0}$ corresponds to the algebraic closure of $K$
in $L$. If we assume (replacing $\phi$ by some $\phi^{r}$) that
$L$ is a field, then since $L$ is a finitely generated field extension
of $K$, the algebraic closure of $K$ in $L$ is a finite extension.
Under Proposition \ref{prop:Properties of K}(iii) we conclude that
$K$ must be algebraically closed in $L$, hence $\mathcal{G}$ must
be connected.
\end{thm}

\subsection{An example\label{subsec:elliptic example}}

Let $K$ and $\phi$ be as in case (Ell). Let $A\in GL_{2}(K)$ be
the matrix
\[
A=\left(\begin{array}{cc}
q & g_{q}(z,\Lambda)\\
0 & 1
\end{array}\right)
\]
where $g_{q}(z,\Lambda)=\zeta(qz,\Lambda)-q\zeta(z,\Lambda)\in K.$
The field
\[
E=K(z,\zeta(z,\Lambda))
\]
is a PV extension for the system $\phi(Y)=AY,$ and
\[
U=\left(\begin{array}{cc}
z & \zeta(z,\Lambda)\\
0 & 1
\end{array}\right)
\]
is a fundamental matrix. Note that $\phi$ induces an automorphism
of $E$. The difference Galois group is
\[
\mathcal{G}=\left\{ \left(\begin{array}{cc}
\alpha & \beta\\
0 & 1
\end{array}\right)|\,\alpha,\beta\in\mathbb{C}\right\} 
\]
and the unipotent subgroup ($\alpha=1$) corresponds, in the Galois
correspondence, to $K(z).$ The field $K(\zeta(z,\Lambda))$ is also
a $\phi$-subfield corresponding to the torus ($\beta=0)$, but is
not a normal extension of $K.$ Note that this field, unlike $E$,
depends on the lattice $\Lambda$.

\section{\label{sec:PPV}Review of parametrized Picard-Vessiot theory}

\subsection{$\psi$-linear algebraic groups}

\subsubsection{Generalities}

In this section $C$ will be an algebraically closed field of characteristic
0, equipped with an\emph{ }automorphism $\psi$ (denoted $\sigma$
in most of the references). For example, if $K$ is a $(\phi,\psi)$-field
and $C=K^{\phi}$ then $C$ inherits an action of $\psi,$ although
this action might well be trivial, as it is in the case (2Ell). We
let $\mathrm{Alg}_{C}^{\psi}$ denote the category of $C$-$\psi$-algebras\emph{}\footnote{\emph{It is important, when developing the general formalism, to abandon
the requirement that $\psi$ be invertible on a general $C$-$\psi$-algebra.
Thus while we maintain the assumption that $\psi$ is an automorphism
of $C$, hence $(C,\psi)$ is ``inversive'', we must allow rings
in which $\psi$ is only an endomorphism, perhaps not even injective,
in the category $\mathrm{Alg}_{C}^{\psi}.$}}. All our schemes and $\psi$-schemes will be affine. If $R\in\mathrm{Alg}_{C}^{\psi}$
we denote by $Spec^{\psi}(R)$ the functor
\[
Spec^{\psi}(R):\mathrm{Alg}_{C}^{\psi}\to\mathrm{Sets}
\]
defined by $Spec^{\psi}(R)(S)=\mathrm{Hom}_{C}^{\psi}(R,S)$ (homomorphisms
of $C$-$\psi$-algebras). Note that if $h\in Spec^{\psi}(R)(S)$
then $\psi(h)=\psi\circ h=h\circ\psi\in Spec^{\psi}(R)(S)$ as well,
so the functor factors through the category of $\psi$-sets.

Let $L$ be a $\psi$-field extension of $C$. A subset $\{a_{1},\dots,a_{n}\}$
of $L$ is called $\psi$\emph{-algebraically independent} over $C$
if the collection \{$\psi^{i}a_{j}|\,0\le i,\,\,1\le j\le n\}$ is
algebraically independent over $C$. The $\psi$-transcendence degree
of $L$ over $C$, denoted $\psi\mathrm{tr.deg.}(L/C)$, is the cardinality
of a maximal $\psi$-algebraically independent set in $L.$ This notion
is well defined (any two such maximal sets have the same cardinality).

We refer to appendix A of \cite{DV-H-W1} for an introduction to $\psi$-schemes
and $\psi$-group schemes. A $\psi$-\emph{algebraic group} $G$ over
$C$ is a $\psi$-group scheme that is $\psi$\emph{-algebraic} over
$C$. This means that its coordinate ring $C\{G\}$ is finitely $\psi$-generated
over $C:$ it contains a finite set $\{u_{1},\dots,u_{n}\}$ such
that $\psi^{i}(u_{j})$ for $0\le i$ and $1\le j\le n$ generate
$C\{G\}$ as a $C$-algebra. It is called $\psi$-\emph{reduced} if
$\psi$ is injective on $C\{G\},$ \emph{perfectly $\psi$-reduced
}if the equation $u^{e_{0}}\psi(u)^{e_{1}}\cdots\psi^{m}(u)^{e_{m}}=0$
($e_{i}\ge0$) forces $u=0$ in $C\{G\},$ and $\psi$\emph{-integral}
if $C\{G\}$ is an integral domain and $\psi$ is injective.

If $G$ is a $\psi$-algebraic group over $C$ its $\psi$-dimension
is a non-negative integer, defined in Definition $A.25$ of \cite{DV-H-W1}.
If $G$ is $\psi$-integral then
\[
\psi\dim(G)=\psi\mathrm{tr.deg.}(Quot(C\{G\})/C).
\]

If $\mathcal{G}$ is a (classical) algebraic group over $C$ then
the functor $B\mapsto\mathcal{G}(B^{\flat})$ from $\mathrm{Alg}_{C}^{\psi}$
to $\mathrm{Groups},$ where $B^{\flat}$ is the $C$-algebra $B$
with the $\psi$-structure forgotten, is representable by a $\psi$-algebraic
group that we denote $[\psi]\mathcal{G}.$ Suppose $\mathcal{G}$
is exhibited as a closed subgroup of $GL_{n,C},$ so that
\begin{equation}
\mathcal{G}=Spec(C[X_{ij},\det(X)^{-1}]/I)\label{eq:Hopf algebra}
\end{equation}
where $1\le i,j\le n$ and $I$ is a Hopf ideal. Then $G=[\psi]\mathcal{G}=Spec^{\psi}C\{G\}$
where $C\{G\}=C[\psi^{k}X_{ij},\det\psi^{k}(X)^{-1}]/I_{\psi}$, $0\le k<\infty,$
the $\psi^{k}X_{ij}$ are symbols treated as independent variables,
the $\psi$-action is the obvious one, and $I_{\psi}$ is the Hopf
$\psi$-ideal generated by all $\psi^{k}(h)$ for $h\in I$. As might
be expected,
\[
\psi\dim([\psi]\mathcal{G})=\dim(\mathcal{G}).
\]

As a non-trivial example of a $\psi$-algebraic group, consider the
$\psi$-closed subgroup of $[\psi]\mathbb{G}_{m}$ given by the equation
\begin{equation}
X^{e_{0}}\psi(X)^{e_{1}}\cdots\psi^{m}(X)^{e_{m}}=1\label{eq:=00005Cpsimonomials}
\end{equation}
for some $e_{i}\in\mathbb{Z}$ (here $\mathbb{G}_{m}=Spec(C[X,X^{-1}])$.

By Lemma A.40 of \cite{DV-H-W1} any closed $\psi$-subgroup of $[\psi]\mathbb{G}_{m}$
is defined by a (possibly infinite) collection of $\psi$-monomials
of the form (\ref{eq:=00005Cpsimonomials}). All that we shall need
is the following weaker result.
\begin{lem}
\label{lem:subgruops G_m}If $G\subsetneq[\psi]\mathbb{G}_{m}$ is
a proper closed $\psi$-subgroup of the multiplicative group, then
there exists a non-trivial $\psi$-monomial of the form (\ref{eq:=00005Cpsimonomials})
satisfied by $G$, and $\psi\dim(G)=0.$
\end{lem}

\subsubsection{(Classical) Zariski closure}

Let $\mathcal{G}$ be a (classical) linear algebraic group over $C$
and $G\subseteq[\psi]\mathcal{G}$ a $\psi$-closed subgroup. We say
that $G$ is \emph{Zariski dense} in $\mathcal{G}$ if for any proper
subvariety (or subgroup, since $G$ is a group functor, this will
turn out to be the same) $\mathcal{H}\subset\mathcal{G}$, $G\nsubseteq[\psi]\mathcal{H}.$
If $\mathcal{G}$ is a subgroup of $GL_{n,C}$ given by the Hopf algebra
(\ref{eq:Hopf algebra}), and 
\[
G=Spec^{\psi}(C[\psi^{k}X_{ij},\det\psi^{k}(X)^{-1}]/J)
\]
for a Hopf $\psi$-ideal $J$, a necessary and sufficient condition
for $G$ to be Zariski dense in $\mathcal{G}$ is that $J\cap C[X_{ij},\det(X)^{-1}]=I,$
i.e. the ``ordinary'' equations, not involving $\psi$, in the ideal
defining $G$, are just those defining $\mathcal{G}$. In such a case
$I_{\psi}\subset J$, because $J$ is a $\psi$-ideal and $I_{\psi}$
is the smallest $\psi$-ideal containing $I$, but this inclusion
might be strict if $G\subsetneq[\psi]\mathcal{G}.$

Conversely, starting with a $\psi$-algebraic group $G$ presented
in the above form, and defining
\[
J\cap C[X_{ij},\det(X)^{-1}]=:I
\]
one sees that $I$ is Hopf ideal in $C[X_{ij},\det(X)^{-1}]$ and
the algebraic group $\mathcal{G}$ that it defines is the (classical)
\emph{Zariski closure} of $G$. 

\subsubsection{The structure of $\psi$-linear algebraic groups whose Zariski closure
is simple}

We shall need the following result, which, for simplicity, we only
state in the case: $C=\mathbb{C}$ and $\psi$ is the identity. It
shows that if $\mathcal{G}$ is \emph{simple}, proper $\psi$-subgroups
of $[\psi]\mathcal{G},$ which are nevertheless Zariski dense in it,
are (under a mild technical condition of $\psi$-reducedness) of a
very special form. Alternatively, one may start with an arbitrary
$\psi$-reduced $\psi$-linear group $G$ and ask (i) that it be \emph{properly
contained} in its (classical) Zariski closure, i.e. $G$ ``is not
classical'', and (ii) that this Zariski closure be simple\footnote{In fact, allowing $m=0$ and $\alpha=1$ in the Proposition, one may
drop (i) and assume only that the Zariski closure of $G$ is simple.}.
\begin{prop}[\foreignlanguage{american}{\cite{DV-H-W2}, Proposition A.19 and Theorem A.20}]
\label{prop:structure when Zariski closure simple} Let $\mathcal{G}$
be a \emph{simple} linear algebraic group over $\mathbb{C}$ and $G\subsetneq[\psi]\mathcal{G}$
a \emph{proper}, $\psi$-closed, $\psi$-reduced subgroup of $[\psi]\mathcal{G}.$
Assume that $G$ is Zariski dense in $\mathcal{G}.$ Then there exists
an automorphism $\alpha\in\mathrm{Aut}(\mathcal{G})$ and an integer
$m\ge1$ such that
\[
G(B)=\{g\in\mathcal{G}(B)|\,\psi^{m}(g)=\alpha(g)\}
\]
for every $B\in\mathrm{Alg}_{\mathbb{C}}^{\psi}.$ Furthermore, replacing
$m$ by a multiple of it we find that there exists an $h\in\mathcal{G}(\mathbb{C})$
such that
\[
G(B)\subset\{g\in\mathcal{G}(B)|\,\psi^{m}(g)=hgh^{-1}\}.
\]
\end{prop}

\subsection{Parametrized Picard-Vessiot extensions}

Let $(K,\phi$) be an inversive $\phi$-field, and assume that it
is endowed with another automorphism $\psi$, commuting with $\phi.$
Assume that the field of $\phi$-constants $C=K^{\phi}$ is algebraically
closed of characteristic 0, and note that it inherits a structure
of a $\psi$-field. Let $(M$,$\Phi)$ be a $\phi$-module of rank
$n$ over $K$, and let $A\in GL_{n}(K)$ be the matrix associated
with it in some basis.
\begin{defn}
A $\psi$\emph{-Picard-Vessiot extension, }called also a \emph{parametrized
Picard-Vessiot (PPV) extension }for $(M,\Phi)$ (or the system $\phi(Y)=AY$),
is a $\phi$-pseudofield $L_{\psi}$ containing $K$ which:

(i) carries, in addition, a structure of a (not necessarily inversive)
$\psi$-field, commuting with $\phi$, 

(ii) trivializes $M$ after base-change, and if $U\in GL_{n}(L_{\psi})$
is a fundamental matrix for $\phi(Y)=AY,$
\[
L_{\psi}=K(u_{ij})_{\psi}=K(\psi^{k}(u_{ij}))
\]
is generated as a total ring (a ring in which every non-zero divisor
is invertible) by the elements $\psi^{k}(u_{ij})$ for $1\le i,j\le n$
and $0\le k<\infty.$ We shall express this property by saying that
$L_{\psi}$ is the $\psi$\emph{-hull} (as a total ring) of $K(u_{ij}).$

(iii) $L_{\psi}^{\phi}=K^{\phi}=C.$
\end{defn}

Here are the main facts about PPV extensions.
\begin{itemize}
\item A PPV extension $L_{\psi}$ as above exists (\cite{O-W}, Theorem
2.28 and Corollary 2.29). This is tricky! One is inclined to construct
inductively (classical) PV extensions for the $\phi$-modules
\[
M_{d}=M\oplus M^{(\psi)}\oplus\cdots\oplus M^{(\psi^{d})}
\]
and go to the limit when $d\to\infty.$ The difficulty is in showing
that we can get $L_{\psi}$ to be a \emph{finite} product of fields.
One should keep track of the number of connected components in this
inductive procedure, and prove that it stays bounded.
\item Let 
\[
R_{\psi}=K[u_{ij},\det(U)^{-1}]_{\psi}=K[\psi^{k}(u_{ij}),\psi^{k}(\det U)^{-1}]
\]
be the \emph{ring} $\psi$\emph{-hull} of $K[u_{ij},\det(U)^{-1}]$
inside $L_{\psi}.$ Then $R_{\psi}$ is $\phi$-simple and $L_{\psi}$
is its total ring of fractions. One calls $R_{\psi}$ the PPV \emph{ring}
of $M$. Since $U$ is uniquely determined up to right multiplication
by $V\in GL_{n}(C),$ $R_{\psi}$ is uniquly determined as a subring
of $L_{\psi}.$
\item Let $L=K(u_{ij})$ and $R=K[u_{ij},\det(U)^{-1}]$ (inside $L_{\psi}$).
Then $L$ is a (classical) PV extension and $R$ a PV ring for $M$.
\end{itemize}

\subsection{The parametrized difference Galois group}

\subsubsection{General facts and definitions}

Assumptions and notation as above, fix a PPV extension $L_{\psi}$
and the PPV ring $R_{\psi}\subset L_{\psi}.$ Consider the functor
$G:\mathrm{Alg}_{C}^{\psi}\rightsquigarrow\mathrm{Groups}$ given
by
\[
G(B)=\mathrm{Aut}_{\phi,\psi}((R_{\psi})_{B}/K_{B})=\mathrm{Aut}_{\phi,\psi}((L_{\psi})_{B}/K_{B}),
\]
the automorphisms of $(R_{\psi})_{B}=B\otimes_{C}R_{\psi}$ that fix
$B\otimes_{C}K$ pointwise and commute with both $\phi$ and $\psi$.
Here $B$ is given the trivial $\phi$-action. If $\sigma\in G(B)$
then
\[
\sigma(U)=UV(\sigma)
\]
with $V(\sigma)\in GL_{n}(B)$. Moreover, since $\sigma$ commutes
with $\psi,$ we have for every $i\ge0$
\begin{equation}
\sigma(\psi^{i}(U))=\psi^{i}(U)\psi^{i}(V(\sigma)).\label{eq:=00005Cpsicompatibilities}
\end{equation}
Thus the choice of $U$ determines an embedding $G\hookrightarrow[\psi]GL_{n,C}.$

The main facts about $G$, mirroring the facts listed for the classical
difference Galois group $\mathcal{G}$, are the following (see \cite{O-W},
section 2.7).
\begin{itemize}
\item $G$ is representable by a $\psi$-linear algebraic group.
\item The Hopf $\psi$-algebra $C\{G\}$ of $G$ is $(R_{\psi}\otimes_{K}R_{\psi})^{\phi}.$
\item The natural map 
\[
R_{\psi}\otimes_{C}C\{G\}=R_{\psi}\otimes_{K}K\{G\}\simeq R_{\psi}\otimes_{K}R_{\psi}
\]
sending $r\otimes h$ ($r\in R_{\psi},\,\,h\in C\{G\}$) to $r\otimes1\cdot h$
is an isomorphism of $K$-$(\phi,\psi)$-algebras. This means that
$W_{\psi}=Spec^{\psi}(R_{\psi})$ is a $\psi$-$G_{K}$-torsor.
\item If $L_{\psi}$ is a field, $\psi\dim(G)=\psi\mathrm{tr.deg.}(L_{\psi}/K)$.
\item The fixed field of a PPV extension under the parametrized Galois group
being defined in the same way as the fixed field of a PV extension
under the classical Galois group, we have
\[
L_{\psi}^{G}=K.
\]
\item More generally, there is a 1-1 correspondence between $\psi$-algebraic
subgroups of $G$ and intermediate $\phi$-pseudofields of $L_{\psi}$
stable under $\psi$. Normal $\psi$-subgroups correspond to PPV extensions
(of some other $\phi$-modules $M')$.
\end{itemize}

\subsubsection{Relation between $G$ and $\mathcal{G}$}

Let $L=K(u_{ij})\subset L_{\psi}$ be the classical PV extension inside
the PPV extension, and $R\subset R_{\psi}$ the PPV ring. Let $\mathcal{G}$
be the classical Galois group. The realization of $\sigma\in G(B)$
as $V(\sigma)\in GL_{n}(B)$ via its action on $U,$ namely 
\[
\sigma:U\mapsto UV(\sigma)
\]
shows that $\sigma$ restricts to an automorphism of $L_{B}$ over
$K_{B},$ hence a map of functors
\[
G\hookrightarrow[\psi]\mathcal{G},
\]
which is evidently injective. In general, it need not be an isomorphism,
as $\sigma\in[\psi]\mathcal{G}(B)=\mathcal{G}(B^{\flat})$ ``does
not know'' about the extra automorphism $\psi$, and may not extend
to $L_{\psi}$ so that the extra compatibilities (\ref{eq:=00005Cpsicompatibilities})
are satisfied. However, since 
\[
C[\mathcal{G}]=(R\otimes_{K}R)^{\phi}\hookrightarrow(R_{\psi}\otimes_{K}R_{\psi})^{\phi}=C\{G\}
\]
is injective, any function from $C[\mathcal{G}]$ that vanishes on
$G$ is 0. It follows that there does not exist a proper (classical)
subgroup $\mathcal{H}\subset\mathcal{G}$ with $G\subset[\psi]\mathcal{H},$
hence
\begin{itemize}
\item $G$ is Zariski dense in $\mathcal{G}$ (\cite{A-D-H-W}, Proposition
2.1.2).
\end{itemize}

\subsubsection{A Galoisian criterion for being $\psi$-isomonodromic}

The $\psi$-Galois group $G$ of a difference $\phi$-module $M$
enables us to state a criterion for $M$ to be $\psi$-isomonodromic,
i.e. for $M\simeq M^{(\psi)}.$
\begin{prop}[Theorem 2.55 of \cite{O-W}]
\label{prop:isomonodromic criterion} The $\phi$-module $M$ is
$\psi$-isomonodromic if and only if there exists an $h\in GL_{n}(C)$
such that
\[
\psi(X)=hXh^{-1}
\]
holds in $G$ (i.e. for any $B\in\mathrm{Alg}_{C}^{\psi}$ and any
$\sigma\in G(B)\subset\mathcal{G}(B)\subset GL_{n}(B),$ $\psi(\sigma)=h\sigma h^{-1}$).
\end{prop}

\begin{proof}
Assume that $M$ is $\psi$-isomonodromic. Then there exists a matrix
$A_{\psi}\in GL_{n}(K)$, such that with $A_{\phi}=A$, we have the
compatibility relation
\[
\phi(A_{\psi})A_{\phi}=\psi(A_{\phi})A_{\psi}.
\]
Using this relation and the relation $1=\phi(U)^{-1}A_{\phi}U$ we
see that
\[
h=\psi(U)^{-1}A_{\psi}U
\]
is fixed under $\phi,$ hence belongs to $GL_{n}(L_{\psi}^{\phi})=GL_{n}(C)$.
Let $\sigma\in G(B)$ and compute $\sigma(\psi(U))$ in two ways.
On the one hand
\[
\sigma(\psi(U))=\sigma(A_{\psi}Uh^{-1})=A_{\psi}UV(\sigma)h^{-1}.
\]
On the other hand
\[
\sigma(\psi(U))=\psi(\sigma(U))=\psi(UV(\sigma))=A_{\psi}Uh^{-1}\psi(V(\sigma)).
\]
Comparing the two expressions we get $\psi(V(\sigma))=hV(\sigma)h^{-1}.$
This string of identities can be reversed. Starting with $h$ as above
and defining $A_{\psi}=\psi(U)hU^{-1}$, we see that $A_{\psi}$ is
fixed by every $\sigma$ in the Galois group, hence lies in $GL_{n}(K)$,
and we get the desired compatibility relation between $A_{\psi}$
and $A_{\phi}.$
\end{proof}
\begin{rem}
The last proof can be given a matrix-free version. If $h:M\simeq M^{(\psi)}$
is an isomorphism of $\phi$-modules, then $h$ can be base-changed
to $L_{\psi}$ and then, since it commutes with $\Phi$, induces an
isomorphism between the modules of solutions. If $\sigma\in G$ (and
not only in $\mathcal{G})$ then $\sigma$ induces a commutative diagram
\[
\begin{array}{ccc}
M_{L_{\psi}}^{\Phi} & \overset{h}{\to} & (M_{L_{\psi}}^{(\psi)})^{\Phi}\\
\sigma\downarrow &  & \downarrow\psi(\sigma)\\
M_{L_{\psi}}^{\Phi} & \overset{h}{\to} & (M_{L_{\psi}}^{(\psi)})^{\Phi}
\end{array}
\]
yielding the relation $\psi(\sigma)=h\sigma h^{-1}.$ If we identify,
as usual, the spaces of solutions $M_{L_{\psi}}^{\Phi}$ and $(M_{L_{\psi}}^{(\psi)})^{\Phi}$
with $\mathbb{C}^{n}$, in the bases given by the columns of $U$
and $\psi(U),$ then $h$ becomes a matrix in $GL_{n}(\mathbb{C})$
as in the Proposition. Conversely, a descent argument shows that given
such a diagram relating the spaces of solutions (\emph{after} base
change to $L_{\psi}$) yields an isomorphism $M\simeq M^{(\psi)}$
(\emph{before} the base-change). In fact, this ``conceptual proof''
is not any different than the ``matrix proof'' by Ovchinnikov and
Wibmer. Unwinding the arguments, one sees that the two proofs are
one and the same.
\end{rem}

\subsection{Example \ref{subsec:elliptic example} continued}

Suppose we add, in Example \ref{subsec:elliptic example}, a second
difference operator $\psi f(z)=f(pz),$ as in the case (2Ell). Then
the $\phi$-module corresponding to the system $\phi(Y)=AY$ is $\psi$-isomonodromic,
and the corresponding system $\psi(Y)=BY$ is given by
\[
B=\left(\begin{array}{cc}
p & g_{p}(z,\Lambda)\\
0 & 1
\end{array}\right).
\]
The compatability relation
\[
\phi(B)A=\psi(A)B
\]
is satisfied. The field $E$ is also a PPV extension, being $\psi$-stable.
The $\psi$-Galois group $G$ is ``classical'', i.e.
\[
G=[\psi]\mathcal{G}.
\]

\section{Some preliminary results}

\subsection{Isomonodromy and solvability}

Let $(K,F,\phi,\psi)$ be as in the case (2Ell) and assume that $2\le p,q$
and $(p,q)=1.$ Let $M$ be a $\phi$-module over $K$, $A$ the associated
matrix (in a fixed basis), $R$ a PV ring for $M$, $L=Quot(R)$ the
corresponding PV extension, $U\in GL_{n}(R)$ a fundamental matrix,
and $\mathcal{G}\subset GL_{n,\mathbb{C}}$ the difference Galois
group, its embedding in $GL_{n,\mathbb{C}}$ determined by the choice
of $U$. The following theorem will be used in the proof of Theorem
\ref{thm:Main Theorem B}, but has independent interest.
\begin{thm*}[= Theorem \ref{thm:isomonodromy and solvability}]
 Assume that $M$ is $\psi$-isomonodromic. Then $\mathcal{G}$ is
solvable.
\end{thm*}
\begin{proof}
By Theorem \ref{thm:structure group and the Galois group}(i), it
is enough to show that with respect to a suitable basis of $M$ the
matrix $A$ is upper triangular. Indeed, if this is the case, take
$H$ to be the Borel subgroup of upper triangular matrices. Since
(a conjugate of) $\mathcal{G}\subset H$ and $H$ is solvable, $\mathcal{G}$
is solvable too.

Endow $M=(M,\Phi,\Psi)$ with a $(\phi,\psi)$-module structure. Call
$M$ solvable if there exists a sequence
\[
0\subset M_{1}\subset\cdots\subset M_{n}=M
\]
of $(\phi,\psi)$-submodules $M_{i},$ $\mathrm{rk}(M_{i})=i$. This
would clearly imply that $A=A_{\phi}$ is gauge-equivalent to a matrix
in upper triangular form.

We show that $M$ is solvable. By induction, it is enough to show
that $M$ contains a rank-one ($\phi,\psi)$-submodule. We apply Theorem
35 of \cite{dS2}. Using the notation there, let $(r_{1},\dots,r_{k})$
be the type of $M,$ $r_{1}\le r_{2}\le\cdots\le r_{k},$ $\sum_{i=1}^{k}r_{i}=n.$
Let $e_{1},\dots,e_{n}$ be the basis of $M$ in which $A$ has the
form prescribed by that theorem. Recall that if we write $A=(A_{ij})$
in block-form, $A_{ij}\in M_{r_{i}\times r_{j}}(K)$, then
\[
A_{ij}(z)=U_{r_{i}}(z/p)T_{ij}(z)U_{r_{j}}(z)^{-1}.
\]
The marix $U_{r}(z)$ is unipotent upper-triangular. The matrix
\[
T_{ij}=\left(\begin{array}{cc}
0 & T_{ij}^{*}\\
0 & 0
\end{array}\right)
\]
where $T_{ij}^{*}$ is a square upper-triangular $s\times s$ marix
for $s\le\min(r_{i},r_{j})$, with \emph{constant} (i.e. $\mathbb{C}$-valued)
entries. An analogous description, with a constant matrix $S$ replacing
$T$, gives the matrix $B,$ associated with $\Psi$ in the same basis.

Let
\[
i_{1}=1,\,\,i_{2}=r_{1}+1,\dots,i_{k}=r_{1}+\cdots+r_{k-1}+1
\]
be the first indices in each of the $k$ blocks. Let $M'=Span_{K}\{e_{i_{1}},\dots,e_{i_{k}}\}.$
Then $M'$ is a rank $k$ $(\phi,\psi)$-submodule of $M$. Moreover,
$\Phi|_{M'}$ and $\Psi|_{M'}$ are given in this basis by \emph{constant}
matrices. In other words, $M'$ descends to $\mathbb{C},$ $M'=M'_{0}\otimes_{\mathbb{C}}K,$
where $M'_{0}$ is a $\mathbb{C}$-representation of $\Gamma=\left\langle \phi,\psi\right\rangle \simeq\mathbb{Z}^{2}$,
and $\Phi$ and $\Psi$ are extended semilinearly from $M'_{0}$ to
$M'$. Since any two commuting endomorphisms of a $\mathbb{C}$-vector
space have a common eigenvector, $M'$ has a rank-one $(\phi,\psi)$-submodule
$M_{1}\subset M'\subset M$, which concludes the proof.
\end{proof}
Incidentally, note that we have given an affirmative answer to Problem
36 in \cite{dS2}.

\subsection{\label{subsec:A-periodicity-theorem}A periodicity theorem}

In this subsection we generalize Theorem 1.1 of \cite{dS1}. Let $\mathscr{D}$
be the $\mathbb{Q}$-vector space of discretely supported functions
$f:\mathbb{C}\to\mathbb{Q},$ i.e. functions for which $\mathrm{supp}(f)=\{z|\,f(z)\ne0\}$
has no accumulation point in $\mathbb{C}.$ For any lattice $\Lambda\subset\mathbb{C}$
let $\mathscr{D}_{\Lambda}$ be the subspace of $f\in\mathscr{D}$
which are $\Lambda$-periodic. We may identify
\[
\mathscr{D}_{\Lambda}=\mathrm{Div}(E_{\Lambda})_{\mathbb{Q}}
\]
with the group of $\mathbb{Q}$-divisors on the elliptic curve $E_{\Lambda}.$

Given two functions $f,\widetilde{f}\in\mathscr{D}$ we say that $\widetilde{f}$
is a \emph{modification at 0} of $f$ if $\widetilde{f}(z)=f(z)$
for every $z\ne0.$

Let $2\le p,q\in\mathbb{N}$ be relatively prime integers: $(p,q)=1.$
Consider the operators
\[
\phi f(z)=f(qz),\,\,\,\psi f(z)=f(pz)
\]
on $\mathscr{D}.$ These operators preserve every $\mathscr{D}_{\Lambda}.$
\begin{prop}
\label{prop:periodicity theorem}Let $f\in\mathscr{D}$. Assume that
for some $\Lambda$-periodic $f_{p},f_{q}\in\mathscr{D}_{\Lambda}$
the relations
\[
f_{q}(z)=f(qz)-f(z)
\]
\[
f_{p}(z)=\sum_{i=0}^{m}e_{m-i}f(p^{1-i}z)
\]
($e_{i}\in\mathbb{Q},$ $e_{m}=1,$ $e_{0}\ne0$) hold for all $z\ne0$.
Then, after replacing $\Lambda$ by a sublattice, a suitable modification
$\widetilde{f}$ of $f$ at $0$ is $\Lambda$-periodic.
\end{prop}

Theorem 1.1 of \cite{dS1} concerned the case $f_{p}(z)=f(pz)-f(z).$
In this case, or more generally if $\sum_{i=0}^{m}e_{m-i}=0$, we
can not forgo the need to modify $f$ at 0 because if two $f$'s agree
outside 0, they yield the same $f_{p}$ and $f_{q}$. 

We shall now show how to modify the proof to treat the more general
case given here.

Observe first that for some $r_{\nu}\in\mathbb{Q}$, $r_{1}=1,$ we
have for every $z\ne0$
\begin{equation}
f(z)=\sum_{\nu=1}^{\infty}f_{q}(\frac{z}{q^{\nu}})=\sum_{\nu=1}^{\infty}r_{\nu}f_{p}(\frac{z}{p^{\nu}}).\label{eq:recursion}
\end{equation}

Formally, this is clear for the first sum, and in the second sum one
solves recursively for the $r_{\nu}$. Since all our functions are
discretely supported, for any given $z$ the infinite sums are actually
finite, and the formal identity becomes an equality.

Let $S_{p}\subset\mathbb{C}/\Lambda$ and $S_{q}\subset\mathbb{C}/\Lambda$
be the supports of $f_{p}$ and $f_{q}$ (modulo $\Lambda$). Let
$\pi_{\Lambda}:\mathbb{C}\to\mathbb{C}/\Lambda$ be the projection
and $\widetilde{S}_{p}=\pi_{\Lambda}^{-1}(S_{p}),\,\,\,\widetilde{S}_{q}=\pi_{\Lambda}^{-1}(S_{q}).$
Let $\widetilde{S}$ be the support of $f$, and $S=\pi_{\Lambda}(\widetilde{S}).$
By (\ref{eq:recursion}) we have
\begin{equation}
\widetilde{S}-\{0\}\subset\bigcup_{\nu=1}^{\infty}p^{\nu}\widetilde{S}_{p}\cap\bigcup_{\nu=1}^{\infty}q^{\nu}\widetilde{S}_{q}.\label{eq:supports}
\end{equation}

\begin{lem}
The set $S$ is finite.
\end{lem}

\begin{proof}
See Lemma 2.3 of \cite{dS1}. It is enough to assume here that $p$
and $q$ are multiplicatively independent.
\end{proof}
\begin{lem}
Let $z\in\widetilde{S}_{q},$ $z\notin\mathbb{Q}\Lambda,$ and let
$n_{q}(z)$ be the largest $n\ge0$ such that $q^{n}z\in\widetilde{S}_{q}$
(it exists since $S_{q}$ is finite and the points $q^{n}z$ have
distinct images modulo $\Lambda$). Note that $n_{q}(z)=n_{q}(z+\lambda)$
for $\lambda\in\Lambda$ so that
\[
n_{q}=1+\max_{z\in\widetilde{S}_{q},\,\,z\notin\mathbb{Q}\Lambda}n_{q}(z)
\]
exists. Then
\[
f(z+q^{2n_{q}}\lambda)=f(z)
\]
for every $z\notin\mathbb{Q}\Lambda$ and $\lambda\in\Lambda$.
\end{lem}

\begin{proof}
The proof preceding Proposition 2.4 in \cite{dS1} holds, word for
word, except that the $P$ there is our $q$ here. Thus, away from
torsion points of the lattice, $f$ is $q^{2n_{q}}\Lambda$-periodic.
The proof of this Lemma, which relies on the previous one, still assumes
only the multiplicative independence of $p$ and $q$.
\end{proof}
We now treat torsion points $z\in\mathbb{Q}\Lambda$, for which we
have to assume $(p,q)=1.$ We may assume, without loss of generality,
that $f,$ hence $f_{p}$ and $f_{q}$, are supported on $\mathbb{Q}\Lambda$,
because away from $\mathbb{Q}\Lambda$ we have already proved periodicity,
so we may subtract the part of $f$ supported on non-torsion points
from the original $f$ without affecting the hypotheses.

Since $S$ is finite we may rescale $\Lambda$ and assume that $f$
is supported on $pq$$\Lambda.$ Then $f_{p}$ is supported on $q\Lambda$
and $f_{q}$ on $p\Lambda$, as becomes evident from (\ref{eq:recursion}).
\begin{lem}
If both $f_{p}$ and $f_{q}$ are $N\Lambda$-periodic, then so is
a suitable modification of $f$ at 0.
\end{lem}

\begin{proof}
The proof of Proposition 2.2 in \cite{dS1} works the same, substituting
the relations (\ref{eq:recursion}) for the relations used there.
\end{proof}
This concludes the proof of Proposition \ref{prop:periodicity theorem}.

\section{Proof of the main theorems}

\subsection{Deduction of Theorem \ref{thm:Main Theorem A} from Theorem \ref{thm:Main Theorem B}}

From now on we let $(K,F,\phi,\psi)$ be as in the case (2Ell), assuming
that $2\le p,q,\,\,\,(p,q)=1.$ Assume that Theorem \ref{thm:Main Theorem B}
is proven, and let $f$ and $g$ be as in Theorem \ref{thm:Main Theorem A}.
Let $n$ be the first integer such that
\[
\psi^{n}(g)\in K(g,\psi(g),\dots,\psi^{n-1}(g)).
\]
Clearly all the $\psi^{i}(g),$ $i\ge n,$ also belong there.

If $g$ were algebraic over $K$, so would be all the $\psi^{i}(g),$
and the field
\[
K(g)_{\psi}=K(g,\psi(g),\dots,\psi^{n-1}(g))
\]
would be a finite extension of $K$ to which $\psi$ extends as an
endomorphism. In fact, since $\psi$ is an automorphism of $K$ and
$[K(g)_{\psi}:K]<\infty,$ it would be an \emph{automorphism} of $K(g)_{\psi}$.
By Proposition \ref{prop:Properties of K}(iii) this is impossible.
Hence $g$ is transcendental over $K$.

Suppose $f$ and $g$ were algebraically dependent over $K.$ Then
this dependence must involve $f$, hence $f$ is algebraic over $K(g)_{\psi},$
and so would be all the $\psi^{i}(f).$ It follows that
\[
\mathrm{tr.deg.}(K(f,g)_{\psi}/K)=\mathrm{tr.deg.}(K(g)_{\psi}/K)\le n<\infty.
\]
A fortiori,
\[
\mathrm{tr.deg.}(K(f)_{\psi}/K)<\infty,
\]
contradicting the conclusion of Theorem \ref{thm:Main Theorem B}.

\subsection{First order equations}

Consider the difference equation
\begin{equation}
\phi(y)=ay\label{eq:order one}
\end{equation}
with $a\in K^{\times}.$ The associated $\phi$-module is $M=Ke$
with $\Phi(e)=a^{-1}e.$ Let $L_{\psi}$ be a PPV extension for $M$,
and $u\in L_{\psi}$ a solution: $\phi(u)=au.$ Replacing $\phi$
and $\psi$ by some powers $\phi^{r}$ and $\psi^{s}$ we may assume
that $L_{\psi}$ is a field. Indeed, let
\[
L_{\psi}=L_{1}\times\cdots\times L_{r}
\]
be the decomposition of the $\phi$-pseudofield $L_{\psi}$ into a
product of fields. Then $\phi^{r}$ is an endomorphism of $L_{1},$
and some power $\psi^{s}$ of $\psi$ must also preserve it, and induces
an endomorphism of $L_{1}$. The subfield of $L_{1}$ generated by
(the projection of) $u$ and all its $\psi^{s}$-transforms is a PPV
extension for $M$ as a $\phi^{r}$-module, which is stable by $\psi^{s}.$

Assume therefore that $L_{\psi}$ is a field.
\begin{prop}
\label{prop:order one equations}The following are equivalent:

(i) $u$ is $\psi$-algebraic over $K,$ i.e. $\mathrm{tr.deg.}(L_{\psi}/K)<\infty.$

(ii) $u$ satisfies an equation $\psi(u)=\widetilde{a}u$ for some
$\widetilde{a}\in K^{\times}.$

(iii) The $\phi$-module $M$ descends to $\mathbb{C}:$ there exist
$b\in K^{\times}$and $c\in\mathbb{C}^{\times}$ such that
\[
a=c\frac{\phi(b)}{b}.
\]
\end{prop}

\begin{cor}
If $ord_{0}(a)\ne0$ then $u$ is $\psi$-transcendental over $K$.
\end{cor}

\begin{proof}
In this case, (iii) can not hold, so (i) can not hold either.
\end{proof}
In \cite{dS1} we proved $(ii)\Rightarrow(iii),$ but we shall not
need this step here. Clearly $(ii)\Rightarrow(i)$, and $(iii)\Rightarrow(ii)$
because if (iii) holds we may assume, replacing $u$ by $u/b$, that
$a\in\mathbb{C}^{\times}.$ Then
\[
(u^{\psi-1})^{\phi-1}=(u^{\phi-1})^{\psi-1}=a^{\psi-1}=1,
\]
so $\widetilde{a}=u^{\psi-1}\in L_{\psi}^{\phi}=\mathbb{C}$, and
(ii) holds. (If we did not assume $a\in\mathbb{C}^{\times}$ we would
only get $\widetilde{a}\in K^{\times}$.) We shall now prove $(i)\Rightarrow(iii)$.
\begin{proof}
Let $G$ be the $\psi$-Galois group of (\ref{eq:order one}). It
is a $\psi$-closed subgroup of $[\psi]\mathbb{G}_{m}.$ If $R_{\psi}=K[u,u^{-1}]_{\psi}\subset L_{\psi}$
is the PPV ring then for any $B\in\mathrm{Alg}_{\mathbb{C}}^{\psi}$
and $\sigma\in G(B)$ we embed $\sigma\mapsto v_{\sigma}\in B^{\times}$
where $\sigma(u)=uv_{\sigma}.$

Assume that $u$ is $\psi$-algebraic. Then
\[
\psi\dim(G)=\psi\mathrm{tr.deg.}(L_{\psi}/K)=0<1=\psi\dim([\psi]\mathbb{G}_{m}),
\]
so $G$ is a proper closed $\psi$-subgroup of $[\psi]\mathbb{G}_{m}.$
It follows from Lemma \ref{lem:subgruops G_m} that there is a $\psi$-monomial
relation
\[
\mu(v_{\sigma})=v_{\sigma}^{e_{0}}\psi(v_{\sigma})^{e_{1}}\cdots\psi^{m}(v_{\sigma})^{e_{m}}=1
\]
($e_{i}\in\mathbb{Z},$ $e_{m}\ne0$) that holds for all $\sigma\in G$.

Let $B\in\mathrm{Alg}_{\mathbb{C}}^{\psi}$ and $\sigma\in G(B).$
Then
\[
\sigma(\mu(u))=\mu(\sigma(u))=\mu(uv_{\sigma})=\mu(u)
\]
so $b'=\mu(u)\in L_{\psi}^{G}=K.$ We conclude that
\begin{equation}
\mu(a)=\mu(u^{\phi-1})=\mu(u)^{\phi-1}=\phi(b')/b'.\label{eq:monomial relation on a}
\end{equation}

Let $\alpha(z)=ord_{z}(a)$ and $\beta(z)=ord_{z}(b').$ These are
the divisors of the elliptic functions $a$ and $b'$, so for some
$\Lambda\subset\Lambda_{0}$ we have $\alpha,\beta\in\mathscr{D}_{\Lambda}.$
Furthermore, taking the divisor of the relation (\ref{eq:monomial relation on a})
gives
\[
\sum_{i=0}^{m}e_{i}\psi^{i}(\alpha)=\phi(\beta)-\beta.
\]
Define
\[
\delta(z)=\sum_{\nu=1}^{\infty}\alpha(\frac{z}{q^{\nu}})\,\,\,(z\ne0);\,\,\,\,\delta(0)=0.
\]
Then $\delta\in\mathscr{D}$ and for $z\ne0$
\[
\begin{cases}
\begin{array}{c}
\alpha(z)=\delta(qz)-\delta(z)\\
\beta(z)=\sum_{i=0}^{m}e_{i}\delta(p^{i}z).
\end{array}\end{cases}
\]
Applying the Periodicity Theorem \ref{prop:periodicity theorem} to
$f_{p}(z)=\beta(pz),$ $f_{q}(z)=\alpha(p^{m}z)$ and $f(z)=\delta(p^{m}z)$
we conclude:
\begin{itemize}
\item After replacing $\Lambda$ by a sublattice, a suitable modification
$\widetilde{\delta}$ of $\delta$ at 0 is $\Lambda$-periodic.
\item We must have $\alpha(0)=0.$
\end{itemize}
The first assertion is the Periodicity Theorem. For the second, if
$z\ne0$ we have
\[
\alpha(z)=\delta(qz)-\delta(z)=\widetilde{\delta}(qz)-\widetilde{\delta}(z).
\]
Let $\Lambda$ be a periodicity lattice for both $\alpha$ and $\widetilde{\delta}.$
Take $0\ne\lambda\in\Lambda$. Then $\alpha(\lambda)=\widetilde{\delta}(q\lambda)-\widetilde{\delta}(\lambda)=0,$
hence $\alpha(0)=0.$

We may therefore assume that $\delta(z)$ is already periodic and
$\alpha(z)=\delta(qz)-\delta(z)$ holds everywhere, including at 0.
Observe, however, that in the process of modifying $\delta$ at 0,
we might no longer have $\delta(0)=0.$

Let $\Pi$ be a fundamental parllelopiped for $\mathbb{C}/\Lambda$
where $\Lambda$ is a periodicity lattice for $\alpha,\beta$ and
$\delta$. Then
\[
0=\sum_{z\in\Pi}\alpha(z)=\sum_{z\in\Pi}\delta(qz)-\sum_{z\in\Pi}\delta(z)=(q^{2}-1)\sum_{z\in\Pi}\delta(z),
\]
so $\delta\in\mathrm{Div}^{0}(\mathbb{C}/\Lambda)$. We also have
\[
\sum_{z\in\Pi}z\alpha(z)=q^{-1}\sum_{z\in\Pi}qz\delta(qz)-\sum_{z\in\Pi}z\delta(z)
\]

\[
=q^{-1}\sum_{z\in q\Pi}z\delta(z)-\sum_{z\in\Pi}z\delta(z)=(q-1)\sum_{z\in\Pi}z\delta(z).
\]
In the last step we used the fact that $q\Pi$ is the union of $q^{2}$
translates $\Pi+\lambda$ and that
\[
\sum_{z\in\Pi+\lambda}z\delta(z)=\sum_{z\in\Pi}z\delta(z)
\]
because $\delta$ is $\Lambda$-periodic and $\sum_{z\in\Pi}\delta(z)=0.$

Let $\Lambda'=(q-1)\Lambda$. Let $\Pi'$ be a fundamental parllelopiped
for $\Lambda'$. Then by the same argument as above
\[
\sum_{z\in\Pi'}z\delta(z)=(q-1)^{2}\sum_{z\in\Pi}z\delta(z)=(q-1)\sum_{z\in\Pi}z\alpha(z)\in\Lambda'
\]
because by Abel-Jacobi $\sum_{z\in\Pi}z\alpha(z)\in\Lambda$.

Replacing $\Lambda$ by $\Lambda'$ we conclude, again by Abel-Jacobi,
that $\delta$ is the divisor of some $b\in K_{\Lambda}.$ Let $c=a/(\phi(b)/b).$
Then $c$ is $\Lambda$-elliptic and its divisor is
\[
\alpha-(\phi(\delta)-\delta)=0.
\]
Thus $c$ is constant, and the proof is complete.
\end{proof}
\begin{cor}[First order case of Theorem \ref{thm:Main Theorem B}]
 \label{prop:main theorem order one, homogenous} Let $f\in F$ satisfy
the first order, linear homogenous equation
\[
\phi(f)=af
\]
with $a\in K^{\times}$. Then either $f\in S$ or $\{f,\psi(f),\psi^{2}(f),\dots\}$
are algebraically independent over $K$.
\end{cor}

\begin{proof}
We may embed $K(f)_{\psi}$ in the PPV extension $L_{\psi}.$ If $\{f,\psi(f),\psi^{2}(f),\dots\}$
are not algebraically independent over $K$ then, according to the
last Proposition $f$ satisfies also a linear homogenous $\psi$-difference
equation over $K.$ By Theorem \ref{thm:dS2} we must have $f\in S$.
\end{proof}
\begin{rem*}
According to \cite{dS1}, in the order one case, if $f$ is $\psi$-algebraic
over $K,$ we can even infer that for some $m\in\mathbb{Z},$ $z^{m}f\in K$.
\end{rem*}
Generalizing from first order homogenous equations to first order
inhomogenous equations is done exactly as in \cite{A-D-H-W}, Proposition
4.5. We do not repeat the proof, as it can be duplicated word for
word, and will not be needed in the sequel, see the remark below.
The only difference is that at the last step in \cite{A-D-H-W}, assuming
$f$ was $\psi$-algebraic over $K$, Theorem 1.1 of that paper was
invoked to deduce that $f\in K.$ Here we should invoke Theorem \ref{thm:dS2}
instead, so we can only deduce $f\in S$. We arrive at the following
Proposition.
\begin{prop}
\label{prop:Main theorem first order} Let $f\in F$ satisfy the inhomogenous
difference equation
\[
\phi(f)=af+b
\]
with $a,b\in K$. Then either $f\in S$ or $\{f,\psi(f),\psi^{2}(f),\dots\}$
are algebraically independent over $K$.
\end{prop}

\begin{rem*}
We shall not need this Proposition. In the last stage of our proof
of Theorem \ref{thm:Main Theorem B} we shall deal with the same type
of inhomogenous equation, but where $b\in S.$ We shall give full
details there.
\end{rem*}

\subsection{The case of a simple $\mathcal{G}$ }

\subsubsection{Recall of notation and assumptions}

Let 
\[
M=(K^{n},\Phi),\,\,\Phi(v)=A^{-1}\phi(v)
\]
be the rank-$n$ $\phi$-module over $K$ associated with the system
$\phi(Y)=AY.$ 

Let $L\subset L_{\psi}$ be PV and PPV extensions for $M$, $\mathcal{G}$
the (classical) difference Galois group and $G$ the (parametrized)
$\psi$-Galois group. 

Fixing a fundamental matrix $U$ with entries in $L$ we get embeddings
\[
G\subset[\psi]\mathcal{G}\subset[\psi]GL_{n,\mathbb{C}}.
\]
When we base change $M$ to $L$ we get the full, $n$-dimensional,
complex vector space of ``solutions''
\begin{equation}
\mathcal{V}=M_{L}^{\Phi}=U\mathbb{C}^{n}\subset M_{L}=L^{n}.\label{eq:solutions}
\end{equation}
If instead of $L$ we use the even larger PPV extension $L_{\psi}$
we get the same complex vector space $\mathcal{V}$, as all the solutions
already lie in $L^{n}$. However, over $L_{\psi}$ we get also the
solution spaces $\psi^{i}(\mathcal{V})$ of all the $\psi$-transforms
$M^{(\psi^{i})},$ $i\ge0.$

The difference Galois group $\mathcal{G}$ acts on $\mathcal{V}$.
If $\sigma\in\mathcal{G}(\mathbb{C})$ then for $v\in\mathbb{C}^{n}$
\[
\sigma(Uv)=UV(\sigma)v,
\]
so $\sigma\mapsto V(\sigma)$ is the matrix representation of $\mathcal{G}$
on $\mathcal{V}$ in the basis consisting of the columns of $U$.

\subsubsection{The simple case}
\begin{lem}[\cite{A-D-H-W}, Lemma 3.9]
\label{lem:connected and reduced} Assume that $L_{\psi}$ is a field.
Then $\mathcal{G}$ is connected and $G$ is $\psi$-integral, hence
in particular $\psi$-reduced.
\end{lem}

The proof relies on Proposition \ref{prop:Properties of K}. Recall
that being $\psi$-integral means that the coordinate ring $\mathbb{C}\{G\}$
is a domain and $\psi$ is injective on it.
\begin{prop}[\cite{A-D-H-W}, Proposition 4.11]
 Assume that $L_{\psi}$ is a field. If $\mathcal{G}$ is simple,
then $G=[\psi]\mathcal{G}.$ In particular,
\[
\psi\mathrm{tr.deg.}(L_{\psi}/K)=\psi\dim(G)=\dim(\mathcal{G})>0.
\]
\end{prop}

\begin{proof}
We repeat the arguments of \cite{A-D-H-W}. Note first that $G$ is
Zariski dense in $\mathcal{G}$ (always true) and $\psi$-reduced
(by the previous lemma). By Proposition \ref{prop:structure when Zariski closure simple},
if $G\subsetneq[\psi]\mathcal{G}$, there exists an $h\in GL_{n}(\mathbb{C})$
and an $m\ge1$ such that
\[
\psi^{m}(X)=hXh^{-1}
\]
holds in $G.$

By Proposition \ref{prop:isomonodromic criterion} $M$ is then $\psi^{m}$-isomonodromic.

By Theorem \ref{thm:isomonodromy and solvability}, $\mathcal{G}$
must be solvable, contradicting the assumption that it was simple.
This contradiction shows that we must have had $G=[\psi]\mathcal{G}.$
\end{proof}
\begin{prop}[\cite{A-D-H-W}, Proposition 4.12]
\label{prop:If U =00005Cpsi algebraic then G solvable} Assume that
$L_{\psi}$ is a field. Then either $\mathcal{G}$ is connected and
solvable or $\psi\mathrm{tr.deg.}(L_{\psi}/K)>0.$
\end{prop}

\begin{proof}
Same as in \cite{A-D-H-W}. Connectendness follows from Lemma \ref{lem:connected and reduced}.
If $\mathcal{G}$ is not solvable then it has a simple quotient $\mathcal{G}/\mathcal{N}.$
The Galois correspondence theorem is used to obtain a PV extension
$L'=L^{\mathcal{N}}$ for a $\phi$-module $M'$, whose difference
Galois group is $\mathcal{G}/\mathcal{N}.$ The PPV extension of the
same $M'$ is a subfield $L'_{\psi}\subset L_{\psi},$ and by the
previous Proposition $\psi\mathrm{tr.deg.}(L_{\psi}'/K)>0.$ A fortiori,
$\psi\mathrm{tr.deg.}(L_{\psi}/K)>0.$
\end{proof}

\subsection{Conclusion of the proof: Galois acrobatics}

We prove the following claim, which clearly implies Theorem \ref{thm:Main Theorem B},
by induction on $n$. The assumptions on $p$ and $q$ are maintained.
\begin{claim}
Let $n\ge1,$ $A\in GL_{n}(K)$ and $u=\,^{t}(u_{1},\dots,u_{n})\in F^{n}$
a solution of $\phi(Y)=AY.$ Assume that all the coordinates $u_{i}$
are $\psi$-algebraic, i.e. $\mathrm{tr.deg.}(K(u)_{\psi}/K)<\infty$
where $K(u)_{\psi}$ is the field $\psi$-hull of $K(u)$ in $F$.
Then $u\in S^{n}.$
\end{claim}

The case $n=1$ follows from Corollary \ref{prop:main theorem order one, homogenous}.
The following Lemma will be used to reduce the general case to certain
inhomogenous first order equations, albeit with coefficients outside
$K$.
\begin{lem}
Without loss of generality, we may assume:
\begin{enumerate}
\item The PPV extension $L_{\psi}$ is a field, and all its elements are
$\psi$-algebraic over $K$. Equivalently, $\psi\dim(G)=0.$
\item $\mathcal{G}$ is solvable, and the matrices $A$ and $U$ are upper
triangular.
\item The diagonal elements of $A$ are in $\mathbb{C}^{\times}.$
\item $K(u)_{\psi}\subset L_{\psi}$ and the vector $u$ is the last column
of $U$.
\end{enumerate}
\end{lem}

\begin{proof}
As before, replacing $\phi$ and $\psi$ by some powers $\phi^{r}$
and $\psi^{s}$, and $A$ by $A_{[r]}$, we may assume, \emph{not
changing $u$,} that the PPV extension $L_{\psi}$ of $M$ is a field,
$K(u)_{\psi}$ (a priori a subfield of $F$) is embedded as a $(\phi,\psi)$-subfield
of $L_{\psi}$, and that the (classical) difference Galois group $\mathcal{G}$
is connected. Let $U$ be a fundamental matrix, with entries in $L_{\psi}.$
The given vector $u$ is some linear combination of its columns.

Let $L_{\psi}^{a}$ be the subfield of $L_{\psi}$ consisting of elements
that are $\psi$-algebraic. It is stable under $\phi$ and $\psi$,
and also under $G.$

Let $\mathcal{V}=M_{L_{\psi}}^{\Phi}=U\mathbb{C}^{n}\subset L_{\psi}^{n}$
be the space of solutions, and $\mathcal{V}^{a}=\mathcal{V}\cap(L_{\psi}^{a})^{n}.$
Since $\mathcal{V}^{a}$ is fixed by $G$ and $G$ is Zariski dense
in $\mathcal{G}$, the (classical) Galois group $\mathcal{G}$ fixes
$\mathcal{V}^{a}$ as well. Let $H$ be the maximal parabolic subgroup
of $GL_{n,\mathbb{C}}$ which is the stabilizer of $\mathcal{V}^{a}$.
Since $\mathcal{G}\subset H$ we may assume that $A$ and $U$ lie
in $H.$ After a $\mathbb{C}$-linear change of coordinates we may
assume that
\[
A=\left(\begin{array}{cc}
A_{11} & A_{12}\\
0 & A_{22}
\end{array}\right),\,\,\,U=\left(\begin{array}{cc}
U_{11} & U_{12}\\
0 & U_{22}
\end{array}\right)
\]
in block form ($A_{11}$ and $U_{11}$ of size $n_{1}\times n_{1}),$
and the first $n_{1}$ columns of $U$ span $\mathcal{V}^{a}\simeq\mathbb{C}^{n_{1}}.$
Thus $u$ is a $\mathbb{C}$-linear combination of the first $n_{1}$
columns of $U$. Replacing our system of equations by the system $\phi(Y)=A_{11}Y$
we may assume that $n=n_{1},$ $A=A_{11},$ $\mathcal{V}=\mathcal{V}^{a}$
etc., hence $L_{\psi}=L_{\psi}^{a}.$ This proves (1).

By Proposition (\ref{prop:If U =00005Cpsi algebraic then G solvable})
the Galois group $\mathcal{G}$ is solvable. By the Lie-Kolchin theorem
we may assume that $\mathcal{G}$ is contained in the upper-triangular
Borel subgroup of $GL_{n,\mathbb{C}}.$ By Theorem \ref{thm:structure group and the Galois group},
so are $U$ and $A$. This proves (2).

Recall that the vector $u$ is linear combination of the columns of
$U$. If the last non-zero entry of $u$ is $u_{r}$ then we can assume,
by a further change of coordinates, not affecting the upper triangular
form, that $u$ is the $r$th column of $U$. Replacing the system
of equations by the one correspoding to the upper-left $r\times r$
block we may assume that $r=n,$ and $u$ is the last column of $U$.
This is (4).

By induction, we may assume that (3) had been proved for all the diagonal
elements of $A,$ but the first one. Since $u_{11}$ is a solution
of $\phi(y)=a_{11}y$, and is $\psi$-algebraic, Proposition \ref{prop:order one equations}
shows that there exists a $b\in K^{\times}$ such that $a_{11}/(\phi(b)/b)\in\mathbb{C}^{\times}.$
Replacing $A$ by the gauge-equivalent $\phi(P)^{-1}AP,$ where $P=\mathrm{diag}.[b,1,\dots,1],$
we get (3).
\end{proof}
From now on the proof can be concluded either by the methods of \cite{dS2},
or by the Galois theoretic ``acrobatics'' adapted from \cite{A-D-H-W}.
We chose to use the second approach.

\bigskip{}

Assume therefore that we are in the situation of the Lemma. In particular
$u_{i}=u_{in}$ is the $i$th entry in the last column of the fundamental
matrix $U$. The $u_{i}$ lie in $F$, but also in $L_{\psi}$, by
our assumption (4) that $K(u)_{\psi}\subset L_{\psi}$.

We may also assume that $E=Quot(S)$ is contained in $L_{\psi}.$
If this is not the case, augment the matrix $A$ by adding to it along
the diagonal a $2\times2$ block as in example \ref{subsec:elliptic example}.
The fundamental matrix gets augmented by the block
\[
\left(\begin{array}{cc}
z & \zeta(z,\Lambda)\\
0 & 1
\end{array}\right),
\]
hence the PPV extension for the augmented system contains $E$. If
we prove the main theorem for the augmented system, we clearly have
proved it also for the original one.

By induction we may assume that $u_{2},\dots,u_{n}\in S$. We have
to show that $u'=u_{1}\in S.$

This $u'\in F$ satisfies the equation
\[
\phi(u')=au'+b
\]
where $a=a_{11}\in\mathbb{C}^{\times},$ and $b=a_{12}u_{2n}+\cdots+a_{1n}u_{nn}\in S$
(by our induction hypothesis). Let $v=u_{11}$, so that $\phi(v)=av.$

Since $(\phi-a)(u')\in S\subset S_{\phi}(F/K)$ clearly $u'\in S_{\phi}(F/K).$
To conclude the proof we shall show, following the ideas of the proof
of Proposition 4.5 in \cite{A-D-H-W} that
\[
u'\in S_{\psi}(F/K)
\]
as well. Theorem \ref{thm:dS2} will show then that $u'\in S$, as
desired.

Consider the matrix
\[
U'=\left(\begin{array}{cc}
v & u'\\
0 & 1
\end{array}\right)\in GL_{2}(L_{\psi}).
\]
This is a fundamental matrix for the system
\[
\phi(Y)=\left(\begin{array}{cc}
a & b\\
0 & 1
\end{array}\right)Y,
\]
regarded as a system of difference equations over $E$. The field
\[
L'_{\psi}=E(v,u')_{\psi}\subset L_{\psi}
\]
is its PPV extension. Furthermore, the $\psi$-Galois group of the
last system is
\[
G'=Gal^{\psi}(L'_{\psi}/E)\subset\left\{ \left(\begin{array}{cc}
\alpha & \beta\\
0 & 1
\end{array}\right)\in GL_{2}\right\} 
\]
and its intersection with the unipotent radical (where $\alpha=1$),
denoted $G'_{u},$ corresponds via the Galois correspondence to $E(v)_{\psi}:$
\[
G'_{u}=Gal^{\psi}(L'_{\psi}/E(v)_{\psi}).
\]
This is a $\psi$-subgroup of $[\psi]\mathbb{G}_{a}$. By our assumption
that $u'$ is $\psi$-algebraic over $K$, hence clearly over $E$,
\[
\psi\dim G'_{u}=\psi\mathrm{tr.deg.}(L'_{\psi}/E(v)_{\psi})=0.
\]
It follows from Corollary A.3 of \cite{DV-H-W1} that there exists
an $0\ne\mathcal{L}_{1}\in\mathbb{C}[\psi]$ such that for any $B\in\mathrm{Alg}_{\mathbb{C}}^{\psi}$
we have
\[
G'_{u}(B)\subset\left\{ \left(\begin{array}{cc}
1 & \beta\\
0 & 1
\end{array}\right)\in GL_{2}(B)|\,\,\mathcal{L}_{1}(\beta)=0\right\} .
\]
Observe that
\[
\phi(\frac{u'}{v})=\frac{u'}{v}+\frac{b}{av},
\]
and $b/av\in E(v)_{\psi},$ so $L'_{\psi}/E(v)_{\psi}$ is a PPV extension
for $\phi(y)=y+(b/av),$ equivalently for the system
\[
\phi(Y)=\left(\begin{array}{cc}
1 & b/av\\
 & 1
\end{array}\right)Y
\]
over $E(v)_{\psi}.$ The action of $\tau\in G'_{u}(B)$ is given by
\[
\tau(\frac{u'}{v})=\frac{u'}{v}+\beta_{\tau}
\]
where $\beta_{\tau}\in B$ corresponds to the above realization of
$\tau$ as a unipotent $2\times2$ matrix. Indeed, if
\[
\tau(U')=U'\left(\begin{array}{cc}
1 & \beta_{\tau}\\
0 & 1
\end{array}\right)
\]
then $\tau(u')=v\beta_{\tau}+u'$ and $\tau(v)=v$.

It follows that for any $\tau\in G'_{u}(B)$
\[
\tau\left(\mathcal{L}_{1}(\frac{u'}{v})\otimes1\right)=\mathcal{L}_{1}(\tau(\frac{u'}{v}))\otimes1=\mathcal{L}_{1}(\frac{u'}{v}+\beta_{\tau})\otimes1=\mathcal{L}_{1}(\frac{u'}{v})\otimes1,
\]
hence $\mathcal{L}_{1}(u'/v)\in E(v)_{\psi}.$ But
\[
(v^{\psi-1})^{\phi-1}=(v^{\phi-1})^{\psi-1}=a^{\psi-1}=1,
\]
so $d=v^{\psi-1}\in\mathbb{C}$ and $\psi(v)=dv,$ $E(v)_{\psi}=E(v).$
It follows that there exists a second operator $\mathcal{L}_{2}\in\mathbb{C}[\psi],$
easily derived from $\mathcal{L}_{1},$ such that
\[
\mathcal{L}_{2}(u')\in E(v).
\]
This $\mathcal{L}_{2}(u')$ satisfies the equation
\[
\phi(y)=ay+\mathcal{L}_{2}(b)
\]
where we have used the fact that $a$ was constant, and where $\mathcal{L}_{2}(b)\in S.$
By Lemma 4.7 of \cite{A-D-H-W}, with $E$ serving as the base field
\emph{and }the intermediate field, there exists a $g\in E$ with
\[
\phi(g)=ag+\mathcal{L}_{2}(b).
\]

We are indebted to Charlotte Hardouin for pointing out the following
lemma.
\begin{lem}
In fact, $g\in S$.
\end{lem}

\begin{proof}
Let $I=\{s\in S|\,sg\in S\}$ be the ideal of denominators of $g.$
If $s\in I$ then $\phi(s)\in I$ because
\[
\phi(s)g=a^{-1}(\phi(sg)-\phi(s)\mathcal{L}_{2}(b))\in S.
\]
Since $S$ is a simple $\phi$-ring (it is the localization at $z$
of the PV ring $K[z,\zeta(z,\Lambda)]$ associated to the system considered
in Example \ref{subsec:elliptic example}), we must have $1\in I,$
so $g\in S.$
\end{proof}
It follows that $\phi(\mathcal{L}_{2}(u')-g)=a(\mathcal{L}_{2}(u')-g).$
Since also $\phi(v)=av,$ the element $d'=(\mathcal{L}_{2}(u')-g)/v$
is fixed by $\phi,$ hence lies in $\mathbb{C}$, and
\[
\mathcal{L}_{2}(u')=d'v+g.
\]
Since $(\psi-d)d'v=d'(\psi-d)v=0$
\[
(\psi-d)\circ\mathcal{L}_{2}(u')=(\psi-d)(g)\in S.
\]
As ($\psi-d)\circ\mathcal{L}_{2}\in\mathbb{C}[\psi]$ and any element
of $S$ is annihilated by a non-trivial operator from $K[\psi],$
we deduce that $u'\in S_{\psi}(F/K),$ and the proof is complete.

\bigskip{}

\emph{Acknowledgements. }The author would like to thank Charlotte
Hardouin and Thomas Dreyfus for fruitful discussions regarding this
work. He would also like to thank Zev Rosengarten and the participants
of the Kazhdan Seminar at the Hebrew University for carefully and
critically attending his lectures on the subject.

\end{document}